%% file: WARM_infinite_cluster_final.tex
\documentclass[a4]{article}
\usepackage{amsmath,amssymb,amsthm,bbm,latexsym,mathdots}
\usepackage{color}
\usepackage{todonotes}
\usepackage{dsfont}
\usepackage{enumerate}
\usepackage{tikz}
\usepackage{bm}
\usetikzlibrary{arrows,automata,positioning}
\usetikzlibrary{decorations.pathreplacing}
\usepackage{color}
\usepackage{soul}

\RequirePackage[colorlinks,citecolor = blue,urlcolor = blue,linkcolor = blue]{hyperref}
\hypersetup{
colorlinks = true,
citecolor = blue,
urlcolor = blue,
linkcolor = blue,
pdfkeywords = {reinforcement model, P\'olya urn, stochastic approximation algorithm, stable equilibria}, 
pdfpagemode = UseNone
}

\newcommand{\Cr}{\mathcal{C}}
\newcommand{\Col}{\mathrm{Col}}

\newcommand{\mc}[1]{\mathcal{#1}}
\newcommand\1{\mathds1}
\newcommand{\indic}[1]{\1_{\left\{#1\right\}}}

\newcommand{\R}{\mathbf{R}}
\newcommand{\N}{\mathbf{N}}

\renewcommand{\P}{\mathbb P}
\newcommand{\E}{\mathbb E}
\newcommand{\mE}{\mathcal{E}}
\newcommand{\Var}{\mathop{\mathrm{Var}}}
\newcommand{\const}{\mathrm{const}}

\newcommand{\eps}{\varepsilon}

\newcommand{\sss}{\scriptscriptstyle}
\newcommand{\blank}[1]{}

\newtheorem{theorem}{Theorem}
\newtheorem*{theorem*}{Theorem}
\newtheorem{lemma}{Lemma}
\newtheorem{proposition}{Proposition}
\newtheorem{corollary}{Corollary}

\theoremstyle{definition}
\newtheorem{definition}{Definition}
\theoremstyle{remark}
\newtheorem{remark}{Remark}

\newcommand{\ms}{\mathsf}

\newcommand{\vep}{\varepsilon}

\newcommand{\bs}[1]{\boldsymbol{#1}}

\def\a{\alpha}
\def\d{{\rm d}}

\def\e{\varepsilon}
\def\E{\mathbb E}

\def\mc{\mathcal}
\def\ms{\mathsf}
\def\N{\mathbb N}

\def\P{\mathbb P}

\def\R{\mathbb R}

\def\FF{\mc F}

\def\TT{\mc T}

\def\pii{p^{\ms{init}}}

\theoremstyle{plain}

\def\s{\sigma}

\title{WARM percolation on a regular tree\\ in the strong reinforcement regime.}

\begin{document}
\author{Hirsch, C.\footnote{University of Groningen:    \textsf{c.p.hirsch@rug.nl} 
} ,
Holmes, M.\footnote{School of Mathematics and Statistics, the University of Melbourne:   \textsf{holmes.m@unimelb.edu.au}} ,
Kleptsyn, V.\footnote{Univ Rennes, CNRS, IRMAR - UMR 6625, F-35000 Rennes, France:   \textsf{victor.kleptsyn@univ-rennes1.fr}
}
}

%

	\maketitle

\abstract{We consider a class of reinforcement processes, called WARMs, on tree graphs. These processes involve a parameter $\a$ which governs the strength of the reinforcement, and a collection of Poisson processes indexed by the vertices of the graph. It has recently been proved that for any fixed bounded degree graph with Poisson firing rates that are uniformly bounded above, in the very strong reinforcement regime ($\a\gg 1$ sufficiently large depending on the maximal degree), the set of edges that ``survive'' (i.e.~that are reinforced infinitely often by the process) has only finite connected components.

The present paper is devoted to the construction of an example in the opposite direction, that is, with the set of surviving edges having infinite connected components. Namely, we show that for each fixed $\a > 1$ one can find a regular rooted tree and firing rates that are uniformly bounded from above, for which there are infinite components almost surely. Joining such examples, we find a graph (with unbounded degrees) on which for any $\a>1$ almost surely there are infinite connected components of surviving edges.}

\section{Introduction and main result}
Let $G = (V,E)$ be a graph with finite degrees and let $\Lambda$ be a Poisson point process on $V\times [0,\infty)$ with intensity $\lambda_v\in (0,\infty)$ on $\{v\}\times [0,\infty)$. We consider a reinforcement process on the edges of this graph. Namely, to 
every edge $e\in E$ we associate a tally $N_t(e)$ that is a piecewise-continuous function of the time $t$, defined in the following way. 
\begin{itemize}
\item It starts with a tally $N_0(e) = 1$, 
\item Whenever the Poisson clock fires at some vertex $v\in V$ at time $t$, we choose a random edge $e$ among those incident to $v$ (write $e \sim v$) with probability proportional to $N_{t-}(e)^\a$, where $\a > 0$ is a parameter of the model. We then update the tally of the chosen edge $N_t(e) = N_{t-}(e) + 1$ and continue. 
\end{itemize}

 Such a process is called a \emph{WARM process} on the graph $G$ with firing rates $(\lambda_v)_{v \in V}$. If $G$ is an infinite graph then it is not obvious that the process is well-defined. However, if $G$ is of bounded degree (i.e.~the maximal degree is finite) and the $\lambda_v$ are uniformly bounded above then the process is indeed well-defined (see \cite[Section 2]{compass}).

Such processes have been studied extensively in recent years \cite{WARM1,whisker,compass,InfWARM3,InfWARM1}, see also \cite{benj,benj2,Lima} for some similar models. These models can be considered as collections of interacting P\'olya urns, where the interactions are graph-based. One of the quantities of interest is the (random) set of edges that ``survive the competition''. There are two ways of defining it: one can consider either the set of edges that are reinforced infinitely often,
\[\mc E_\infty = \{e\in E:N_t(e) \to \infty\},\]
or the set of those that are reinforced (at least) linearly often:
\[\mc E_+ = \{e\in E: \liminf_{t\to\infty} t^{-1}N_t(e) > 0\}. \]
These two sets actually (almost surely) coincide; in other words, almost surely an edge that is not reinforced linearly often gets reinforced only finitely many times. This can be seen from the general P\'olya urn properties; we refer the reader to~\cite{InfWARM3} for a discussion of these sets. However, we do not need such an equality for the purposes of the present paper, and from now on we will concentrate on the set of linearly reinforced edges~$\mc E_+$.

In \cite{compass} it is shown that if $G$ has bounded degree and the $\lambda_v$ are uniformly bounded above then for all $\a$ sufficiently large $\mc E_\infty$ has no infinite component. A query of a referee for that paper motivated us to find an example of a WARM process with bounded firing rates on a graph of bounded degree for which $\mc E_\infty$ can have an infinite component. {Our main result, Theorem \ref{t:tree} below, concerns the existence of such infinite components on trees.}

{A main ingredient in the proof of these results is the following, that seems to us to be of interest in its own right. It describes the behaviour of a \emph{single} P\'olya urn with the parameter $\a > 1$, starting in a situation akin to a ``cat against a thousand mice'' (that is, in the initial state there are many balls of one colour, competing against single balls of a lot of other colours). Informally speaking, for an appropriate choice of the parameters, the cat will (with high probability) grow into a tiger, while a few mice, even though they lose, will nonetheless grow to the size of a cat.

For convenience, when there are $n + 1$ colours these will be labelled $0,1,\dots, n$ and we will write $(m_0,m_1,\dots, m_n)$ to represent the number of balls of each colour $0,\dots, n$ respectively at some time.

\begin{theorem}\label{t:Polya}
		Fix $\a > 1$. For all sufficiently large $m$ there exists $n = n(\a,m){\ge 5}$, such that the following hold with probability greater than $\frac45$ for an $\a$-P\'olya urn on $n + 1$ colours with initial state $(m,1,1,\dots,1)$:
		\begin{itemize}
			\item only colour 0 is selected more than $2m-1$ times; and 
			\item there are at least $5$ colours selected exactly~$m-1$ times. 
		\end{itemize}
\end{theorem}
\begin{remark}
As the reader will easily see from the proof, the number ``$5$'' of colours in the second conclusion 
can be replaced with any other (arbitrarily large) number, and the probability $4/5$ can be 
replaced with any other $p<1$. We proceed with these fixed numbers only to reduce the 
complexity of the formulae and conclusions, making them a bit more readable.
\end{remark}

In order to use this theorem in the proof of our main results, let us fix how the 
P\'olya urn process above is defined, via a sequence of i.i.d.~standard uniform random variables $(U_i)_{i \in \N}$. The first selection from the urn depends on the value of $U_1$ and subsequently the $i$th selection depends on $U_i$ and the configuration after the previous $i-1$ selections. Namely, starting with initial tallies $(N_0(0),N_0(1),\dots , N_{0}(n)) = (m,1,\dots,1)$ we select colour $i$ at time $j\in \N$ if 
\begin{align}
\label{Upolya}
\dfrac{\sum_{i'<i}N_{j-1}(i')^\a}{\sum_{i'=0}^{n}N_{j-1}(i')^\a}<U_j \le \dfrac{\sum_{i'\le i}N_{j-1}(i')^\a}{\sum_{i'=0}^{n}N_{j-1}(i')^\a}.
\end{align}
Here and elsewhere in the paper, an empty sum is equal to 0 by convention.

\begin{remark}
Our WARM process will be defined on a probability space in terms of independent Poisson clocks located at the vertices $v\in V$ of the graph and i.i.d.~standard uniform random variables $(U_{v,i})_{v \in V,i \in \N}$ from which the edge choice upon the $i$th firing at vertex $v$ is made.
\end{remark}

A standard argument then allows us to deduce from Theorem~\ref{t:Polya} 
that the same conclusions hold after a sufficiently large but finite number of steps:
\begin{corollary}\label{cor:Polya}
For any $\a,m,n$ as in Theorem~\ref{t:Polya} (with $(m,1,...,1)$ as the initial state) there exists $M_0$ such that for any $M\ge M_0$ with probability greater than $\frac45$,
 the values of $U_1,\dots,U_M$ are such that in the first $M$ colour selections:
\begin{quote}
		\begin{itemize}
			\item only colour $0$ is selected more than $2m-1$ times; and 
			\item there are at least $5$ colours selected exactly~$m-1$ times. 
		\end{itemize}
\end{quote}
\end{corollary}
}

To state our main result, Theorem~\ref{t:tree} below, let us first introduce some notation. Let $\TT = \TT_n = (V,E)$ be the rooted $n$-ary tree with the root $x_0$, {where the root has degree $n$ and all other vertices have degree $n+1$.} Let $d:V\to \N\cup \{0\}$ be the distance to the root. Also, let us introduce the notion of a ``finitely-joined subgraph'': 
\begin{definition}
For a graph $G_2 = (V_2,E_2)$ and its subgraph 
\[
G_1 = (V_1,E_1), \quad V_1\subset V_2, \quad E_1\subset E_2, 
\] 
we write $G_1\sqsubset G_2$ if $E_{1,2}: = E_2\setminus E_1$ contains only finitely many edges incident to~$V_1$; in this case, we say that $G_1$ is \emph{a finitely-joined subgraph} of $G_2$. {For WARM processes on $G_1$ and $G_2$ where $G_1\sqsubset G_2$, we write $G_1\overset{\bs{\lambda}}{\sqsubset}G_2$, if the WARM process on $G_2$ has firing rates on $V_1\subset V_2$ equal to those of the WARM process on $G_1$.}
\end{definition}

Our main result is the following theorem.
\begin{theorem}\label{t:tree}$ $
\begin{enumerate}[\normalfont (i)]
\item For any $\a > 1$ there exist $n\in\N$, $q\in (0,1)$, such that for the graph $\TT_n$, equipped with the firing rates $\lambda_v = q^{d(v)}$, \begin{quote} 
the set of linearly reinforced edges $\mE_+ $ for the corresponding WARM process almost surely contains (infinitely many) infinite connected components.
\end{quote}

\hspace{-1cm}
Moreover, the above holds uniformly in a neighbourhood of $\a$, i.e.

	\item For any $\a>1$ there exists a neighbourhood $O_\a$ of $\a$ such that: there exist $n\in\N$, $q\in (0,1)$, such that for the graph $\TT_n$, equipped with the firing rates $\lambda_v = q^{d(v)}$, 
	\begin{quote}
 for all $\a_0\in O_\a$, the set of linearly reinforced edges $\mE_+ $ for the corresponding WARM process almost surely contains (infinitely many) infinite connected components.	
 \end{quote}
\end{enumerate}
Finally, {\normalfont (ii)} above holds for any graph $T'$ such that {$\TT_n\overset{\bs{\lambda}}{\sqsubset} T'$. }
\end{theorem}

It is easy to join bounded-degree graphs together (see e.g.~Figure \ref{f:ex-th-2}) in order to prove the following.
\begin{theorem}\label{t:a}
There exists a connected graph $G = (V,E)$ and a bounded intensity function $\lambda:V\to \R_+ $, such that for any $\a > 1$ the WARM process on $(G,\lambda)$ is well-defined and almost surely the corresponding set $\mE_+ $ of linearly reinforced edges contains (infinitely many) infinite connected components. 
\end{theorem}

\begin{figure}[!h!]
	\begin{tikzpicture}
\node[circle,fill = black,scale = 0.3,label=1] (A0) at (0,0) {};
\node[circle,fill = black,scale = 0.3,label=1] (A1) at (6,0) {};
\node[circle,fill = black,scale = 0.3,label=1] (A2) at (12,0) {};
\node (D1) at (13,0) {$\dots$};
\draw (0,0)--(12.5,0);

\node[circle,fill = black,scale = 0.3,label=left:{$q_1$}] (B11) at (-0.5,-1) {};
\node[circle,fill = black,scale = 0.3,label=left:{$q_1$}] (B12) at (0.5,-1) {};
\draw (A0)--(B11);
\draw (A0)--(B12);

\node[circle,fill = black,scale = 0.3,label=left:{$q_2$}] (B21) at (5,-1) {};
\node[circle,fill = black,scale = 0.3,label=left:{$q_2$}] (B22) at (6,-1) {};
\node[circle,fill = black,scale = 0.3,label=left:{$q_2$}] (B23) at (7,-1) {};
\draw (A1)--(B21);
\draw (A1)--(B22);
\draw (A1)--(B23);

\node[circle,fill = black,scale = 0.3,label=left:{$q_3$}] (B31) at (10.5,-1) {};
\node[circle,fill = black,scale = 0.3,label=left:{$q_3$}] (B32) at (11.5,-1) {};
\node[circle,fill = black,scale = 0.3,label=left:{$q_3$}] (B33) at (12.5,-1) {};
\node[circle,fill = black,scale = 0.3,label=left:{$q_3$}] (B34) at (13.5,-1) {};
\draw (A2)--(B31);
\draw (A2)--(B32);
\draw (A2)--(B33);
\draw (A2)--(B34);

\node[circle,fill = black,scale = 0.3,label=left:{$q_1^2$}] (C111) at (-1,-2) {};
\node[circle,fill = black,scale = 0.3] (C112) at (-0.5,-2) {};
\node[circle,fill = black,scale = 0.3] (C121) at (0.5,-2) {};
\node[circle,fill = black,scale = 0.3,label=right:{$q_1^2$}] (C122) at (1,-2) {};
\draw (B11)--(C111);
\draw (B11)--(C112);
\draw (B12)--(C121);
\draw (B12)--(C122);
\node (LD) at (0,-2.5) {$\vdots$};
\node (LDL) at (-1.25,-2.5) {$\iddots$};
\node (LDR) at (1.25,-2.5) {$\ddots$};

\node[circle,fill = black,scale = 0.3,label=left:{$q_2^2$}] (C211) at (4,-2) {};
\node[circle,fill = black,scale = 0.3] (C212) at (4.5,-2) {};
\node[circle,fill = black,scale = 0.3] (C213) at (5,-2) {};
\node[circle,fill = black,scale = 0.3] (C221) at (5.5,-2) {};
\node[circle,fill = black,scale = 0.3] (C222) at (6,-2) {};
\node[circle,fill = black,scale = 0.3] (C223) at (6.5,-2) {};
\node[circle,fill = black,scale = 0.3] (C231) at (7,-2) {};
\node[circle,fill = black,scale = 0.3] (C232) at (7.5,-2) {};
\node[circle,fill = black,scale = 0.3,label=right:{$q_2^2$}] (C233) at (8,-2) {};
\draw (B21)--(C211);
\draw (B21)--(C212);
\draw (B21)--(C213);
\draw (B22)--(C221);
\draw (B22)--(C222);
\draw (B22)--(C223);
\draw (B23)--(C231);
\draw (B23)--(C232);
\draw (B23)--(C233);
\node (MD) at (6,-2.5) {$\vdots$};
\node (MDL) at (4,-2.5) {$\iddots$};
\node (MDR) at (8,-2.5) {$\ddots$};

\node (RD1) at (10,-1.5) {$\iddots$};
\node (RD1) at (11,-1.5) {$\iddots$};
\node (RD1) at (14,-1.5) {$\ddots$};

\node[circle,fill = black,scale = 0.3,label=left:{$q_3^2$}] (C331) at (11.5,-2) {};
\node[circle,fill = black,scale = 0.3] (C332) at (12,-2) {};
\node[circle,fill = black,scale = 0.3] (C333) at (12.5,-2) {};
\node[circle,fill = black,scale = 0.3,label=right:{$q_3^2$}] (C334) at (13,-2) {};
\draw (B33)--(C331);
\draw (B33)--(C332);
\draw (B33)--(C333);
\draw (B33)--(C334);
\node (RD) at (12.5,-2.5) {$\vdots$};

\node (RDL) at (11.5,-2.5) {$\iddots$};
\node (RDR) at (13.5,-2.5) {$\ddots$};
\end{tikzpicture}

\caption{Example of a graph for Thm.~\ref{t:a}. The firing rates are indicated at various vertices.}\label{f:ex-th-2}
\end{figure}
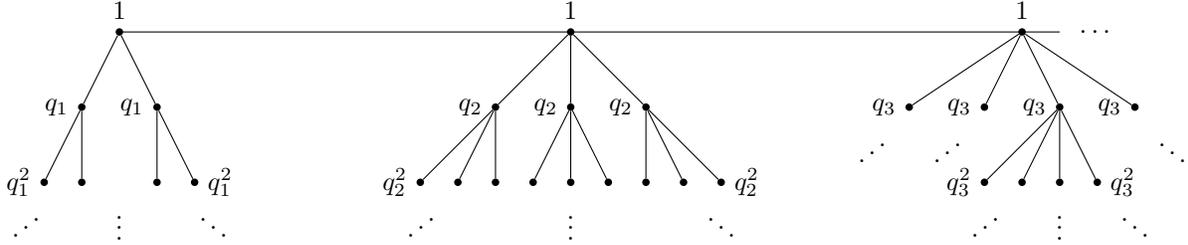

Before we proceed with the detailed proof (assuming Theorem \ref{t:Polya}) of our main result Theorem \ref{t:tree}, it is useful to present sketches of the proofs which indicates the main features of the arguments.

\subsection{Main ideas and plan of the paper}

\begin{proof}[Sketch of the proof of Theorem~\ref{t:Polya}]
	To prove Theorem~\ref{t:Polya} we will use Rubin's construction~\cite{davis}. Namely, take families $(\xi_j,\eta_{j,k})_{j \in \N, k = 1,\dots, n}$ of i.i.d.~exponential random variables with mean 1. These random families can be coupled with the $\a$-P\'olya urn (starting with $m$ balls of colour 0 and one ball of each of the $n$ other colours) in the following way. For $r \in \N$ consider the partial sums
	\begin{equation}
	\label{partials}
	S^{(r)} = \sum_{j = m}^{m+r-1}\frac{\xi_j}{j^\a} \quad \text{ and } Z_k^{(r)} = \sum_{j = 1}^{r}\frac{\eta_{j,k}}{j^\a}.
	\end{equation}
of the series
\begin{equation}\label{eq:sums}
S=\sum_{j = m}^{\infty}\frac{\xi_j}{j^\a} \quad \text{ and } Z_k = \sum_{j = 1}^{\infty}\frac{\eta_{j,k}}{j^\a}
\end{equation}
respectively. 

The P\'olya urn process can then be coupled with this collection in the following way. On the real line one places a ball of colour $0$ in the urn at each time in $(S^{(r)})_{r\in \N}$ and a ball of colour $k$ in the urn at each time in $(Z_k^{(r)})_{r\in\N}$.

The winning colour is the one for which the sum of the full series~\eqref{eq:sums} is smaller, and the final tallies for the losing colours are equal to the initial ones plus the numbers of the respective last partial sums~\eqref{partials} that are less than the full sum of the winning series.
	
	Now, note, that for a large $m$ the series $S=\sum_{j = m}^{\infty}\frac{\xi_j}{j^\a}$ has many terms of roughly the same magnitude, thus (in the spirit of the Law of Large Numbers) its sum will probably be close to its expectation, which itself is close to $\frac1{(\a-1)m^{\a-1}}$; see Proposition~\ref{p:S}.

	To ensure the conclusions of Theorem~\ref{t:Polya}, we have to take the (large) number $n$ of other colours such that:
\begin{itemize}
\item It is sufficiently large, so that with high probability we observe at least $5$ (random) colours $k=k_1,\dots,k_5$ such that $Z_k^{(m-1)}<S<Z_k^{(m)}$.
\item However, it is sufficiently small, so that with high probability $Z_k^{(2m)}>S$ for all~$k$.
\end{itemize}
We are doing this in Section~\ref{s:Polya}.
\end{proof}

\begin{proof}[Sketch proof of Theorem \ref{t:tree}]
Fix $m,n,M$ given by the conclusions of Corollary~\ref{cor:Polya}.
In fact, we will see that the conclusions of Theorem~\ref{t:tree} actually hold for \emph{all} 
sufficiently small~$q > 0$. Namely, we will observe the following process happening recurrently for vertices farther and farther from the root (and thus occurring on larger and larger time-scales); its timeline is illustrated in Fig.~\ref{f:time}.
\begin{enumerate}
\item\label{i:d5} Before the vertex $v$ has first fired, the edge (let's call it $0_v$) from vertex $v$ to its parent $\bar v$ has received $m-1$ reinforcements from its parent, and will receive no more in the {near} future. The firings at vertex $v$ then start to occur, with reinforcements from these firings following the $\a$-power P\'olya scheme (based on i.i.d.~uniform random variables $(U_{v,i})_{i \in \N}$).

Its behaviour is then described by Theorem~\ref{t:Polya}: with probability at least $4/5$, the edge towards its parent ``wins'', becoming 
the only one to be reinforced by $v$ substantially, while at least $5$ 
{edge-offspring (i.e.~edges from a vertex to its children)} 
of $v$ get exactly $m-1$ reinforcements.
		\hspace{-1.5cm}
\item From now on, $v$ reinforces the edge towards its parent, while 
{the children of $v$} 
from step~\ref{i:d5} act each 
by repeating the same scenario, but on the $q$-slowed time scale.

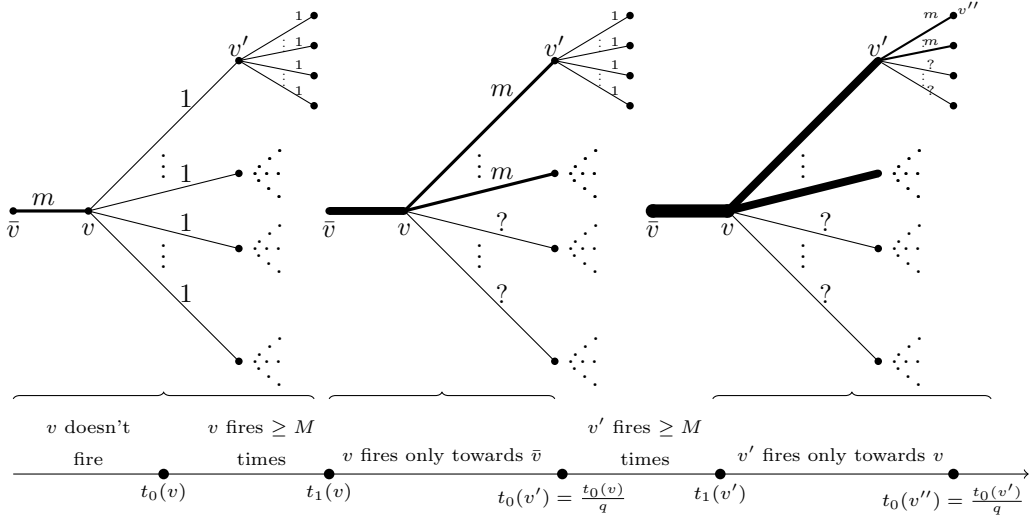
\begin{figure}[!h!]
\input{outline}
\caption{Timeline}\label{f:time}
\end{figure}

\item Finally, after {these children} of $v$ ``raise'' their own edge-offspring to the tally $m$ {(and meanwhile the edges joining them to $v$ get most of the reinforcement from them)}, the configuration starts to converge to an equilibrium. 

Namely, each vertex mainly reinforces the edge to its parent, with a small (but linear in time) portion going to its edge-offspring. The limiting configuration contains a Galton-Watson tree (associated to the behaviour in the P\'olya urns), with the expected number of descendants at least $ \frac45\cdot 5 > 1$. Thus this G-W tree is infinite with positive probability.
\end{enumerate}
Such a branching process actually starts from any given point with a positive (and bounded away from~$0$) probability, thus ensuring the appearance of such a tree (and, moreover, of an infinite number of such trees) almost surely.


We will formalise these arguments in Section~\ref{s:tree}. Proposition~\ref{p:induction} therein states what we need to make an induction argument, comparing a set of surviving edges to a G-W tree. The induction argument itself is performed in Proposition~\ref{p:cryst}; meanwhile, Proposition~\ref{p:somewhere} provides the base for such an induction (roughly speaking, stating, that every part of a tree has a chance to be disconnected from its parents and to become a root of such dynamics). 
\end{proof}

We deduce Theorem~\ref{t:a} from Theorem~\ref{t:tree}. Though the deduction is almost immediate due to the last conclusion of the theorem (and the example shown on Figure~\ref{f:ex-th-2}), one has to check that the WARM process on the constructed graph is well-defined (that is no more granted, as it is no longer a graph with bounded degrees). This is done in Section~\ref{sec:thmsproof}.

\section{The WARM on trees}\label{s:tree}	
\subsection{Initial definitions}
Let $(\Omega, \FF,\P)$ denote a probability space on which for each $v\in V$, ${\bs T}_v = (T_{i,v})_{i \in \N}$ are the points of a Poisson processes with rate $\lambda_v = q^{d(v)}$, and $({\bs T}_v)_{v \in V}$ are independent of each other, and on which $(U_{i,v})_{i \in \N,v\in V}$ are i.i.d.~standard uniform random variables that are independent of $({\bs T}_v)_{v \in V}$. {Let $\bs{U}_v=(U_{i,v})_{i \in \N}$.}

From each non-root vertex $v\in V$, we fix a labelling of the incident edges $e\sim v$ as~$0_v,\dots ,n_v$, where $0_v$ always denotes the edge to the parent of~$v$.

Starting with $N_0(e) = 1$ for every $e \in E$, the process evolves as follows (compare with \eqref{Upolya}): 
At the $j$th firing moment $t = T_{j,v}$ at vertex $v$, given the tallies of edges $\{N_{t-}(e)\}_{e\sim v}$ we select edge $i_v$ if 
\begin{align}
\label{choice}
\dfrac{\sum_{i'<i}N_{t-}(i'_v)^\a}{\sum_{i'=0}^{n} N_{t-}(i'_v)^\a}<U_{j,v} \le \dfrac{\sum_{i'\le i}N_{t-}(i'_v)^\a}{\sum_{i'=0}^{n} N_{t-}(i'_v)^\a},
\end{align}
and we set $N_t(i_v) = N_{t-}(i_v) + 1$ (while $N_t(e) = N_{t-}(e)$ for $i_v\ne e\sim v$).
As almost surely there are no infinite chains of dependence \cite{compass}, this leads to a well-defined process. In particular, \eqref{choice} implies that the edge $0_v$ joining $v$ to its parent is reinforced by the $j$-th firing at $v$ if and only if 
\begin{equation}
\label{Uparent}
0<U_{j,v} { \le } \frac{N_{T_{i,v}-}(0_v)^\a}{\sum_{e\sim v} N_{T_{i,v}-}(e)^\a}.
\end{equation}

Now, for $\a > 1$ let us choose and fix $m$ and $n$ satisfying the conclusions of Theorem~\ref{t:Polya}. Also, let us consider what happens after a large finite number~$M$ of P\'olya urn steps. Namely (see Corollary~\ref{cor:Polya}), for all sufficiently large $M$ with probability greater than $4/5$, after $M$ steps of the $\a$-P\'olya urn process, starting with $(m,1,\dots,1)$ we get $(x_0,x_1,\dots,x_n)$ with 
$$
\forall i = 1,\dots, n \quad x_i\le 2m
$$
and 
$$
\#\{i = 1,\dots, n \mid x_i = m\} \ge 5.
$$
We will fix such a sufficiently large $M$ (satisfying also some other lower bounds that are to be determined later). We also fix a sufficiently small value $\delta_0 \in (0,{\frac 12}) $ such that 
\begin{equation}\label{eq:delta-def}
\left(\frac{3\delta_0}{1-2 \delta_0}\right)^\a < \frac{\delta_0}2.
\end{equation}

To state properties and propositions below we will need, in addition to $n,m$ and $M$, to choose some $\eps,q>0$. The order, in which we choose these parameter values is as follows: Fix $\a>1$. We choose $n,m$ as in Theorem \ref{t:Polya}, and $M\ge M_0$ as in Corollary \ref{cor:Polya}. Then, $\eps\in (0,1)$ is chosen to be sufficiently small, and finally $q>0$ is chosen to be sufficiently small (compare with Lemma~\ref{l:good} below). However, for the intermediate propositions we will immediately assume the following inequalities (that can be guaranteed when the choice is made in this order):
\begin{equation}
\label{Mmn}
M>4mn
\end{equation}
\begin{equation}\label{eq:qeM}
q<\frac{\eps^2}M\wedge \frac{\delta_0 \eps^2}n.
\end{equation}

\begin{definition}
	\label{d:good}
On our probability space $(\Omega, \FF,\P)$, for fixed $m,n,M,\vep,q > 0$, a vertex $v\in V^-:=V\setminus\{x_0\}$ is 
\begin{itemize}
\item \emph{P\'olya-regular}, if the variables $(U_{j,v})_{j=1}^M$ are such that the conclusions of Corollary~\ref{cor:Polya} hold for the discrete-time P\'olya urn process defined from \eqref{Upolya} with the variables $(U_j)_{j = 1}^M$ defined by $U_j: = U_{j,v}$, and with starting configuration $(m,1,\dots, 1)$. In this case, the colours $i_v$ that are reinforced (in this single-urn discrete-time process) exactly $m-1$ times are called its \emph{P-children}; 
\item \emph{P\'olya-non-disturbing}, if 
\begin{equation}\label{eq:non-dist}
\forall j = M + 1,\dots, M', \quad U_{j,v} \le 1 - n\left(\frac{2m + \eps^{-2}qj}{j/2}\right)^\a,
\end{equation}
where $M': = \lceil \eps^{-2} q^{-1} M\rceil $.
\item \emph{P\'olya-well behaved}, if
\begin{equation}
	\label{eq:well}
\forall j > M', \quad \#\Big\{i\in \{M' + 1,\dots,j \}:\, U_{i,v} \ge 1 - \frac{\delta_0}2 \Big\} \le \delta_0 j.
\end{equation}
\item \emph{P\'olya-good} if it is P\'olya-regular, P\'olya-non-disturbing and P\'olya-well behaved;
\item \emph{time-regular}, if for all $k\in \N$ the firing times $T_{k,v}$ satisfy 
\[
\eps \frac k{\lambda_v} < T_{k,v} < \eps^{-1} \frac k{\lambda_v}.
\]
\item \emph{time-good}, if all its children (but not necessarily the vertex itself) are time-regular. 
\item \emph{tree-good}, if it is time-good and P\'olya-good.
\end{itemize}
\end{definition}

\begin{remark}\label{r:independent}
Let $\mc T_v=\mc{T}_v(\a,m,n,M,\vep,q)$ denote the event that vertex $v$ is tree-good. Then, $(\mc T_v)_{v \in V^-}$ are independent events since they depend on disjoint sets of independent random variables ($v$ being tree-good depends on $\bs{U}_v$ and the variables $\bs{T}_{v'}$ for children $v'$ of $v$). 
Similarly, for each $v$, the events that $v$ is P\'olya-regular, time-good, P\'olya-non-disturbing, and P\'olya-well behaved are independent.
\end{remark}

For every vertex $v$ fix the times 
\[
t_0(v): = \frac{\eps}{\lambda_v}, \quad t_1(v): = \frac M{\eps \lambda_v}.
\]
By definition, a time-regular vertex $v$ never fires before $t_0(v)$ and fires at least $M$ times by~$t_1(v)$. Since {
\[\dfrac{M}{\eps q}\le \eps \Big\lceil\dfrac M{\eps^2 q}\Big\rceil=\eps M',\]
we also see that if $v$ is time-regular} then for any time-regular descendant $v'$ of $v$, by the time $t_1(v')$ the vertex $v$ fires at most $M'$ times; this motivates the occurrence of $M'$ in the conditions above.

\begin{definition}
Let $\bar v$ denote the parent of the vertex $v\in V^-$. We call $\bar v$ for $v$ a
\begin{itemize}
\item \emph{nurturing parent}, if by the moment $t_0(v)$ the edge $\bar vv$ is reinforced by $\bar v$ exactly $m-1$ times;
\item \emph{non-disturbing parent}, if between $t_0(v)$ and $t_1(v)$ the edge $\bar vv$ is never reinforced by~$\bar v$.
\item \emph{good parent}, if $\bar v$ is both nurturing and non-disturbing.
\end{itemize}
\end{definition}

\begin{remark}
The property of being a good parent depends on the Poisson clocks and uniform random variables throughout the tree. Nevertheless we will develop an induction that works with \emph{independent} random variables at each vertex: see Proposition~\ref{p:induction}.
\end{remark}

\subsection{The crystallisation tree}

For $v\in V$, let $D(v)$ denote $v$ and all its descendants.

\begin{definition}[Crystallisation tree]
\label{def:crystal}
	The \emph{crystallisation tree} $\Cr_v$ rooted at a tree-good vertex $v \in V^-$ is defined in the following way. We start with the vertex $v$. Then, for each vertex that we have already added,
we append all tree-good $P$-children of that vertex (as well as the edges joining these vertices to their parent).
The procedure is then (possibly infinitely) repeated. If $v$ is not tree-good then the crystallisation tree $\Cr_v$ rooted at $v$ is defined to be empty.
\end{definition}
{
\begin{remark}
\label{rem:crystaldepend}
The crystallisation tree $\Cr_v$ depends only on the variables $(\bs{U}_u)_{u \in D(v)}$ and $(\bs{T}_u)_{u \in D(v)\setminus \{v\}}$
\end{remark}
}

The reason for considering the crystallisation tree can essentially be seen from the following two propositions (whose proofs we are postponing to Sections~\ref{s:proof-cryst} and \ref{s:proof-GW} respectively).
\begin{proposition}
	\label{p:cryst}
Let $v\in V^-$ be a time-regular, tree-good vertex with a good parent~$\bar v$. Then, the crystallisation tree $\Cr_v$ rooted at $v$ is contained in $\mE_+$.
\end{proposition}

\begin{proposition}
\label{p:GW}
For all $\a,m,n$ as in Theorem \ref{t:Polya}, the following holds. 
	For all sufficiently large $M = M(n, \a) \ge 1$ and sufficiently small $\e = \e(M, n, \a) > 0$ and $q = q(\e, M, n, \a) > 0$, conditionally on the event that a vertex $v\in V^-$ is tree-good, its crystallisation tree $\Cr_v$ is a \emph{super-critical} Galton-Watson tree (with one vertex at the initial level).
\end{proposition}

In order to apply Proposition~\ref{p:cryst}, one needs to ensure an appropriate ``starting point'' for such a tree. 
 To this end, let us introduce the following notation: let
\[
B(v):=\{v\} \cup \{u:u \sim v\} 
\] 
be the union of the vertex $v$ with its neighbours, and let 
\[
\FF_v := \s\big({\bs U}_v, (\bs{T}_u)_{u \in B(v)}\big)
\]
be the $\s$-algebra generated by the Poisson clock rings for these vertices and the uniform variables associated to the vertex itself. Also, denote by $D(v)$ the tree of descendants of $v$ (including $v$). Then, for a vertex $v$ and its parent $\bar v$, the events associated to the construction of the crystallisation tree $\Cr_v$ {(that is, the tree goodness of $v$ and its descendants, which depends only on $(T_{j,u})_{j \in \N, u \in D(v)\setminus \{v\}}$ and $(U_{j,u})_{j \in \N, u\in D(v)}$)}, are independent from these belonging to $\FF_{\bar v}$.

For $v\in V$ such that $d(v)\ge 2$, let $A_{\bar v}$ be the event that all the following occur: 
\begin{itemize}
\item no vertices among the neighbours of $\bar v$ fire before $t_0(v)$, {and $v$ is time-regular}
\item $\bar v$ fires by $t_0(v)$ exactly {$m-1$ times}, and after that doesn't fire at all before $t_1(v)$;
\item all these {$m-1$} firings at $\bar v$ reinforce the edge $0_v=\bar v v$
\end{itemize}
\begin{proposition}[Initial spark]
	\label{p:somewhere}
	Let $\a, M,\eps,q$ be fixed.	Then, there exists $\pii> 0$ such that for any $v\in V$ such that $d(v)\ge 2$, 
	\begin{enumerate}[\normalfont (i)]
	\item $A_{\bar v}\in \FF_{\bar v}$, and 
	\item on the event $A_{\bar v}$, $v$ is time-regular and $\bar v$ is a good parent for~$v$, and
	\item
	 \[
\P(A_{\bar v})\ge \pii.
\]
\end{enumerate}
\end{proposition}
\begin{proof}
{The event that \emph{no vertices among the neighbours of $\bar v$ fire before $t_0(v)$} depends only on $(T_{1,u})_{u \in B(\bar v) \setminus \{\bar v\}}$. The event that \emph{$v$ is time-regular} depends only on $\bs{T}_v$ (and on this event $v$ does not fire before $t_0(v)$). The event that \emph{$\bar v$ fires by $t_0(v)$ exactly {$m-1$ times}, and after that doesn't fire at all before $t_1(v)$} depends only on $(T_{i,\bar v})_{i=1}^m$. If these first two events occur then the occurrence of the event that \emph{all these {$m-1$} firings at $\bar v$ reinforce the edge $0_v=\bar v v$} depends only on $(U_{j,\bar v})_{j=1}^{m-1}$ (this condition becomes the condition that a P\'olya urn with the initial condition $(1,\dots, 1)$ defined from the $\bs{U}_{\bar v}$ variables reinforces a given fixed colour on each of the first $m-1$ occasions). Therefore $A_{\bar v}\in \mc{F}_{\bar v}$.

The claim (ii) is immediate from the definition of $A_{\bar v}$ and the definition of good parent.

The probability that the first two events (i.e.~the first two bullet points above) in the definition of $A_{\bar v}$ occur is positive and does not depend on $v$, since the products $t_0(v) \lambda_v=\eps$ and $t_1(v)\lambda_v=\frac{M}{\eps}$ do not depend on $v$.  Meanwhile, conditionally on these events, the last event in $A_{\bar v}$ has positive probability that again does not depend on~$v$ by the above urn observation. Thus $\P(A_{\bar v})>0$ and this probability does not depend on $v$.}
\end{proof}

Next, we need many such ``initial sparks'', where a Galton-Watson tree can be started independently. {To verify this, let us use the labelling scheme where the children of the root are labelled $1,\dots, n$ and for any labelled vertex $v$ of the tree that is not the root, its children are labelled $v1,\dots, vn$.
\begin{lemma}\label{l:independent}
There exist vertices $(v_{(k)})_{k \in \N}$ such that the sets $B(\bar v_{(k)}) \cup D(v_{(k)})$ are disjoint.
\end{lemma}
\begin{proof}
	For $k \in \N$ let $v_{(k)}$ denote the vertex whose label is 
$
1\underbrace{22\dots 2}_{k-1}111.
$	
Then, 
\[
B(\bar v_{(k)}) \cup D(v_{(k)})\subset D(\bar{\bar v}_{(k)}),
\] 
where $\bar{\bar v}_{(k)}$ is the vertex with the label 
$
1\underbrace{22\dots 2}_{k-1}1.
$	
Clearly $D(\bar{\bar v}_{(k)})$ are disjoint for different $k$ since all descendants of $\bar{\bar v}_{(k)}$ have labels starting with the labelling of $\bar{\bar v}_{(k)}$.
\end{proof}
}

To ensure that almost surely there is an infinite number of infinite connected components, we have to ensure that these can be disconnected from each other. This can be achieved relatively easily in the tree.

%
%

For a vertex $v \in V^-$, let $v_n$ denote the last child of $v$. Also let $
E_v:= E_v^1 \cap E_v^2$, where
\begin{equation}\label{eq:E1}
E_v^1:= \bigcap_{j\ge 0} \left\{U_{j+1,v} \le \frac{(1+(j/n))^\a}{1+(1+(j/n))^\a} \right\},
\end{equation}
\begin{equation}\label{eq:E2}
E_v^2:=\bigcap_{j\ge 0} \left\{U_{j+1,v_n} \ge \frac1{1+(1+(j/n))^\a} \right\}.
\end{equation}
Thus, $E_v$ depends only on $\bs{U}_v$ and ${\bs U}_{v_n}$.

\begin{lemma}[Disconnecting]
\label{lem:disconnect}
There exists $p_0>0$ such that for every vertex $v\in V^-$ 
\begin{enumerate}[\normalfont (i)]
\item\label{i:positive} $\P(E_v)\ge p_0$,
\item\label{i:disconnecting} on the event $E_v$ the edge $vv_n$ is \emph{never} reinforced.
\end{enumerate}
\end{lemma}
\begin{proof}
The first claim follows from the fact that we are intersecting independent events of probabilities $(1-O(j^{-\a}))$, with $\a>1$, hence the product of probabilities converges to some positive value $p_0$, that does not depend on the choice of the vertex~$v$.

The second claim is proved by induction, noting that 
	if by the moment of $(j+1)$-st firing at the vertex $v$ the edge $vv_n$ has never been reinforced (neither from $v$, nor from $v_n$), then at least one of the other edges incident to $v$ has tally at least $1+\frac{j}{n}$ (as the vertex $v$ has already fired $j$ times), and the inequality~\eqref{eq:E1} thus implies that the reinforcement will not go to the last edge~$vv_n$. 

	Similarly if by the moment of $(j+1)$-st firing at the vertex $v_n$ the edge $vv_n$ has never been reinforced (neither by $v$, nor by $v_n$), at least one of the other edges starting from $v_n$ has tally at least $1+\frac{j}{n}$ (as the vertex $v_n$ has already fired $j$ times). Again, the inequality~\eqref{eq:E2} thus implies that the reinforcement will not go to the first edge~$vv_n$. 
\end{proof}

\subsection{Proof of Theorems \ref{t:tree} and \ref{t:a}}
\label{sec:thmsproof}
We are now ready to prove our main results (assuming Theorem \ref{t:Polya} and Propositions~\ref{p:cryst} and~\ref{p:GW} that will be proved in the next sections).

\begin{proof}[Proof of Theorem \ref{t:tree}]
	Let $\a>1$ be given. Fix $n,m$ given by Theorem~\ref{t:Polya}.  Further, fix $q, \eps, M$, given by Corollary~\ref{cor:P2}, so that the crystallisation Galton--Watson trees are supercritical. Let $v_{(k)}$ be the vertices given in Lemma \ref{l:independent}, and let $z_{(k)}$ denote the vertex $v_{(k)}n111$, which is a descendant of $v_{(k)}$. For $k \in \N$ consider the events
\begin{align}
J_k:=E_{v_{(k)}}\cap A_{\bar{z}_{(k)}}\cap \{\#\Cr_{z_{(k)}}=\infty\}.
\end{align}	
As noted above Lemma \ref{lem:disconnect}, $E_{v_{(k)}}$ depends only on $\bs{U}_{v_{(k)}}$ and $\bs{U}_{v_{(k)}n}$. The event $A_{\bar{z}_{(k)}}$ depends only on ${\bs U}_{\bar{z}_{(k)}}={\bs U}_{v_{(k)}n11}$ and $(\bs{T}_u)_{u \in B(\bar{z}_{(k)})}$, while the event $\{\#\Cr_{z_{(k)}}=\infty\}$ depends only on $(\bs{U}_u)_{u \in D(z_{(k)})}$ and $(\bs{T}_u)_{u \in D(z_{(k)})\setminus \{z_{(k)}\}}$. Thus these three events are independent.

Now, these events have positive probability due to Lemma 2, Proposition~\ref{p:somewhere}(iii) and Proposition~\ref{p:GW} respectively, and therefore $J_k$ has positive probability. Moreover, the events $(J_k)_{k \in \N}$ are independent for different $k$. Therefore infinitely many $J_k$ occur almost surely. 

By Propositions~\ref{p:cryst} and~\ref{p:somewhere}(ii) the crystallisation trees $\Cr_{z_{(k)}}$ are all subsets of $\mc{E}_+$. They are also disconnected from each other by edges that are never reinforced (these never reinforced edges are obviously in $\mc{E}_+^c$). Since infinitely many $J_k$ occur this implies that there are infinitely many infinite components of $\mc{E}_+$ as claimed in part (i).

Next, note that (ii) follows from the fact that the arguments above for a particular choice of $n,m,q,M,\eps$ are preserved by a small perturbation of~$\a$ (for a fixed $M$ the strict inequality in Corollary~\ref{cor:Polya} is preserved under a small perturbation by continuity).

Finally, note that if $\TT_n\overset{\bs{\lambda}}{\sqsubset} T'$ then since $T'$ contains only finitely many edges that are connected to $\TT_n$ but not in $\TT_n$, the same argument above yields that infinitely of the events $J_k$ occur in that part of $\TT_n$ that is ``after'' these finitely many edges, giving the final claim.
\end{proof}

Let us deduce Theorem \ref{t:a} from Theorem \ref{t:tree}.

\begin{proof}[Proof of Theorem \ref{t:a}]

Part (ii) of Theorem~\ref{t:tree} states that for any $\a\in (1,\infty)$ there exist parameters $(n,q)$ and a neighbourhood $O_\a $ of $\a$ such that the connected components conclusion for the tree $\TT_n$ with $\lambda_v=q^{d(v)}$ holds simultaneously for all $\a'\in O_\a$. Let us extract a countable cover $\{O_{\a_j}\}$ of $(1,\infty)$ by such neighbourhoods, and let $(n_j,q_j)$ be the corresponding parameters. 

Now, take the trees $\TT_{n_j}$ equipped with the respective functions $\lambda_{v,j}:=q_j^{d(v)}$ (where $d$ is the distance to the corresponding root), and let us attach the $j$-th such tree to the point $j$ of the positive half-line. That is, we take a disjoint union
\[
G' := \sqcup_{j \ge 1}\TT_{n_j}
\]
of these trees, equipping it with the function $\lambda_v$ that is the union of the respective functions $\lambda_{v,j}$, and transform it into a connecting graph $G$ by adding an edge from the root of $\TT_{n_j}$ to the root of $\TT_{n_{j+1}}$ (see  Figure~\ref{f:ex-th-2}) for each $j \in \N$.

Then, to complete the proof of Theorem~\ref{t:a} we have only to check that the WARM process on $G$ is well-defined. Indeed, once this is established, for any $\a \in (1,\infty)$ one of the neighbourhoods $O_{\a_j}$ covers~$\a$. By the choice of $O_{\a_j}$, $n_j$ and $q_j$, as $\TT_{n_j}\overset{\bs{\lambda}}{\sqsubset} G$, this implies that the set $\mE_+ $ of linearly reinforced edges almost surely contains (infinitely many) infinite connected components. 

Let us now check that the WARM process is well-defined. Note that we cannot apply \cite[Theorem~1]{compass} directly, as the degrees of the graph $G$ are not bounded. However, as in \cite{compass}, for a WARM process to be well-defined, it suffices to show that almost surely there are no infinite \emph{descending chains}, that is, sequences $(v_j, T_{v_j,k_j})_{j\in\N}$ such that for all $j$ the vertex $v_{j+1}$ is a neighbour of $v_{j}$ and $T_{v_{j+1},k_{j+1}}<T_{v_j,k_j}$. Indeed, the presence of such a chain (in which case the result of $k_j$-th firing at the vertex $v_j$ cannot be resolved without resolving $k_{j+1}$-th firing at the vertex $v_{j+1}$) is the only possible obstacle to the well-definedness of a WARM process; see \cite[Definition~1]{compass} (and the discussion after it).

Assume now that such an infinite chain $(v_j,T_{v_j,k_j})$ exists. Then, either all except for a finite number of vertices $v_j$ are contained in the same tree $\TT_{n}$ for some $n$, or shifts between these trees infinitely many times. In the former case, removing the initial part that is not contained in $\TT_n$, we get an infinite descending chain that consists only of the vertices from this tree. However, this almost surely does not take place; indeed, any of these trees \emph{is} a graph with bounded degrees, hence almost surely contains no infinite descending chains due to~\cite[Proposition~1]{compass}.

On the other hand, if an infinite descending chain changes trees infinitely many times, it contains an infinite sub-chain visiting only the root vertices. This would be an infinite descending chain contained in the half-line graph, that is formed by the roots of all $\TT_n$; again, such an infinite chain almost surely does not exist, as $\N$ is also a graph with bounded degrees. 

As both cases for the infinite descending chain almost surely are impossible, with probability~$1$ there are no such chains, and thus the WARM process on $G$ is well-defined. This concludes the proof of Theorem~\ref{t:a}.
\end{proof}

\subsection{Proof of Proposition \ref{p:cryst}}\label{s:proof-cryst}
Proposition \ref{p:cryst} will be proved by induction.

\begin{proposition}[Main induction step]\label{p:induction}
Assume that the vertex $v$ has a good parent $\bar v$ and that $v$ is tree-good and time-regular. Then: 
\begin{itemize}
\item $v$ is a good parent for all of its P-children, and 
\item 
the edge $0_v=\bar vv$ is reinforced with a frequency at least $\vep \lambda_v/2$, i.e.
\[
\liminf_{t\to\infty} \frac1tN_t(\bar vv) \ge \frac{\eps\lambda_v}2.
\]
\end{itemize}

\end{proposition}

\begin{proof}
Let us describe the evolution of tallies on the edges adjacent to $v$. 
\begin{itemize}
\item \textbf{Before $t_0(v)$.} Note that up to time $t_0(v)$ there were no firings at the vertex $v$ due to its time-regularity, thus, the only reinforcements for the $\bar vv$ edge are those coming from the parent~$\bar v$. Also, the children of $v$ are time-regular (as $v$ is time-good). Hence they also do not fire by the time $t_0(v)$, and moreover, will not fire even until $q^{-1} t_0(v) > t_1(v)$ (the inequality is due to~\eqref{eq:qeM}). Since also $\bar v$ is a nurturing parent for $v$, it reinforces the edge $\bar vv$ exactly $m-1$ times by the time $t_0(v)$. Hence the tallies of edges adjacent to $v$ at time $t_0(v)$ are $(m,1,\dots,1)$. 
\begin{figure}[!h!]
{\center
	\input{neighb}

}
\caption{Neighbourhood of the vertex $v$ after $M$ of its firings; the encircled vertices are time-regular.}
\end{figure}
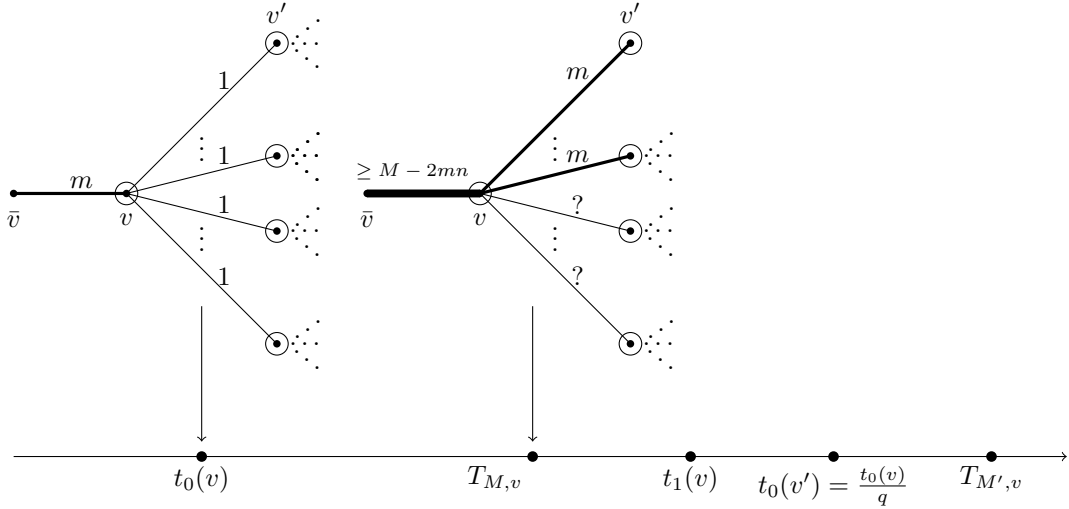
\item \textbf{Between $t_0(v)$ and $T_{M,v}$.} 
Since $v$ is time-regular, $T_{M,v}<t_1(v)$: the vertex $v$ will fire at least $M=\eps \lambda_v t_1(v)$ times by time $t_1(v)$. 

The children of $v$ do not fire before $t_1(v)$ due to their time-regularity. Also, $\bar v$ is a non-disturbing parent for $v$, hence from time $t_0(v)$ until $t_1(v)$ it does not reinforce $\bar v v$. Hence, the only reinforcements that can come to the edges adjacent to~$v$ during this period are those coming from the firing at~$v$. Thus the reinforcement of the edges adjacent to $v$ during this time {are determined by the number of firings at $v$ and the P\'olya urn choices according to the $U_{j,v}$ as in \eqref{choice} with $(m,1,\dots, 1)$ as the initial state.}

Since $v$ is P\'olya-regular, after its first $M$ firings the conclusions of Corollary~\ref{cor:Polya} will hold. In particular, $v$ has at least $5$ $P$-children, i.e.~edge-offspring of $v$ which have been reinforced exactly $m-1$ times after these $M$ firings at~$v$. Also, the tally of the edge $0_v=\bar v v$ at this moment is at least $m+(M-n(2m-1))=M-2mn+m+n$, as at most $(2m-1)n$ firings of the vertex $v$ went to other edges.

\item \textbf{From $T_{M,v}$ to $T_{M',v}$.}
	\emph{All the firings of $v$ from $(M+1)$-st to $M'$-th reinforce the edge $\bar v v$.} {We will prove this by induction on $j\in M+1,\dots, M'$, verifying both the base case $j=M+1$ and the induction step via the same argument. For $j=M+1,\dots, M'$, the tally of the edge $0_v=\bar vv$ just before the $j$-th firing at the vertex $v$ is at least $j-2mn>\frac j2$: for the base case $j=M+1$ this follows from the fact above that the tally of this edge is at least 
	\[M+(m+n)-2mn\ge M+1-2mn>(M+1)/2
	\] 
	after $M$ firings at $v$ (the latter inequality following from \eqref{Mmn}), and then for all the other $j$ this follows trivially from the induction hypothesis.}

On the other hand, the tally of every other edge $v v'$ (joining $v$ to one of its children $v'$) at this moment does not exceed 
\[
N_{T_{j,v}-}(vv') \le {2m}+\eps^{-2} q j;
\] 
the first summand here is an upper bound on the tally at $T_{M,v}$, and the second is an upper bound for the number of firings at~$v'$ at this moment, which holds due to time-regularity of both $v$ and $v'$: 
\[
T_{\lceil \eps^{-2} q j \rceil,v'} > \frac{\eps}{\lambda_{v'}} \eps^{-2} q j = \frac1{q \lambda_v} \eps^{-1} q j > T_{j,v}.
\]
Finally, the P\'olya-non-disturbing condition~\eqref{eq:non-dist} implies that with these tallies, the $j$-th firing will also reinforce the edge $0_v=\bar v v$. Indeed, using \eqref{eq:non-dist} and $1-x<1/(1+x)$ for all $x\ge 0$ we have
\[
U_{j,v} \le 1 - n\left(\frac{2m + \eps^{-2}qj}{j/2}\right)^\a < \frac{(j/2)^\a}{(j/2)^\a + n \cdot (2m+\eps^{-2}qj)^\a} < \dfrac{N_{T_{j,v}-}(0_v)^\a}{N_{T_{j,v}-}(0_v)^\a + \sum_{i' > 0}N_{T_{j,v}-}(i'_v)^\a},
\]
and hence (recall~\eqref{Uparent}) the edge $0_v = \bar vv$ is reinforced by this firing.

The above description already suffices to conclude that the first claim of the proposition holds. Indeed, for any child $v'$ of $v$, we have by \eqref{eq:qeM} that $t_0(v')=q^{-1}t_0(v)>t_1(v)>T_{M,v}$. Therefore since $v$ is time-good and time-regular it has fired more than $M$ times by time $t_0(v')$.  Thus the edge towards any $P$-child $v'$ of $v$ at this moment has a tally equal to~$m$, so $v$ is a nurturing parent for $v'$.  Meanwhile, $T_{M',v}>q^{-1}t_1(v)=t_1(v')$, and thus $v$ does not reinforce the edge $vv'$ between the times $t_0(v')$ and $t_1(v')$, and hence $v$ is a non-disturbing parent for $v'$.

\item \textbf{After $T_{M',v}$}. Let us now proceed to the second claim of the proposition, to the asymptotic reinforcement of the edge $0_v=\bar v v$. We will prove by induction on $j=M'+1,M'+2,\dots$ the following statement: if $U_{j,v}<1-\frac{\delta_0}2$, then the $j$-th firing of the vertex $v$ reinforces~$0_v$. 

	To do so, let us estimate the tallies of the edges adjacent to $v$ up to the moment of the $(j+1)$-st firing. {In the case $j+1=M'+1$, the tally of the edge $0_v$ satisfies
\[
N_{T_{j,v}-}(0_v) \ge (M'-M)+(M-2mn)= M'-2mn,
\]
which is the lower bound for $N_{T_{M',v}}(0_v)$ obtained from the bounds in the previous time periods.  
For all larger $j$, from the induction hypothesis, the tally of the edge $0_v$ satisfies
\[
N_{T_{j,v}-}(0_v) \ge (M'-2mn) + \#\Big\{i\in \{M' + 1,\dots,j \}:\, U_{i,v} < 1 - \frac{\delta_0}2\Big\}.
\]
}
{In each case, due} to~\eqref{eq:well} this sum is lower-bounded by 
\begin{equation}\label{eq:0v}
(M'-2mn)+(j-M'-\delta_0 j) = (1-\delta_0) j-2mn>(1-2\delta_0)j;
\end{equation}
here we have used $\delta_0 j>\delta_0 M'>2mn$, that follows from $M'\ge M/(\eps^{2}q)$, the inequality $M>2mn$ and~\eqref{eq:qeM}.
On the other hand, the sum of tallies of the edges $vv'$, where $v'$ are children of $v$, does not exceed 
\begin{equation}\label{eq:sum-delta}
2mn+n\cdot \eps^{-2}q j+\delta_0 j = (\delta_0+n\eps^{-2}q) j +2mn < 3 \delta_0 j.
\end{equation}
Indeed, the first summand on the left hand side of~\eqref{eq:sum-delta} bounds from above the sum of tallies at the moment $t_0(v')$, the second corresponds to the number of possible reinforcements by $v'$, and the last one to those by $v$. {The inequality in \eqref{eq:sum-delta} arises similarly from \eqref{eq:qeM} and \eqref{Mmn} as above.}

Hence, we have 
\[
\frac{\sum_{i>0} N_{T_{j+1,v}-}(i_v)}{N_{T_{j+1,v}-}(0_v)} \le \frac{n(2m+\eps^{-2}q j)+\delta_0 j}{(1-\delta_0) j-2mn} < \frac{3\delta_0 j}{(1-2 \delta_0) j} =\frac{3\delta_0}{1-2 \delta_0}.
\]
Since $\a > 1$, we obtain
\[
\frac{\sum_{i>0} N_{T_{j+1,v}-}(i_v)^\a }{N_{T_{j+1,v}-}(0_v)^\a } \le \left(\frac{\sum_{i>0} N_{T_{j+1,v}-}(i_v)}{N_{T_{j+1,v}-}(0_v)} \right)^\a  < \left(\frac{3\delta_0}{1-2 \delta_0}\right)^\a  < \frac{\delta_0}2,
\]
where the last inequality is due to~\eqref{eq:delta-def}. Thus, if $U_{j,v}<1-\frac{\delta_0}2$, then due to~\eqref{Uparent} the $(j+1)$-st firing also goes to $0_v$, thus proving { both the base case and} the induction step.

Hence, for each $j\ge M'$, \eqref{eq:0v} is a lower bound for the tally of the edge $0_v$ at time $T_{j,v}$. Due to the time-regularity of the vertex $v$ and the fact that $\delta_0<1/2$ we obtain
\[
\liminf_{t\to\infty} \frac1t N_t(\bar vv) \ge \liminf_{j \to \infty}\dfrac{N_{T_{j,v}}(0_v)}{T_{j+1,v}}\ge \lim_{j \to \infty}\dfrac{(1-\delta_0)j-2mn}{(j+1)/\eps\lambda_v}
= (1-\delta_0) \eps \lambda_v > \frac{\eps}2\lambda_v.
\]
This proves the second claim of the proposition, and thus concludes the proof.
\end{itemize}
\end{proof}

\begin{proof}[Proof of Proposition \ref{p:cryst}]
Let us show that every node $v'$ of the crystallisation tree $\Cr_v$ has the following properties
	\begin{enumerate}
		\item\label{i:family} $v'$ has a good parent, is time-regular and {tree-good};
		\item\label{i:edges} $\mE_+$ contains the edge $0_{v'}$.
	\end{enumerate}
Let us start with the first claim. We proceed by induction on the distance of a vertex $v'$ in the crystallisation tree to $v$. The base $v'=v$ is given by the assumptions of the proposition. Now, assume that this statement holds for some vertex $v'\in \Cr_v$. Then, Proposition~\ref{p:induction} is applicable to $v'$, and thus $v'$ is a good parent for its $P$-children. Thus, any child $v''$ of $v'$ that belongs to $\Cr_v$ has a good parent, is time-regular (because $v'$ is time-good) and is tree-good (otherwise it wouldn't be appended to the crystallisation tree). This proves the induction step, and hence claim~\ref{i:family} for all the vertices in~$\Cr_v$.

Now, let $v'\in \Cr_v$. Again, Proposition~\ref{p:induction} is applicable to $v'$, and from its second conclusion we get that the edge $0_{v'}$ is reinforced with a lower-bounded intensity: 
\[
\liminf_{t\to\infty} \frac1tN_t(0_{v'}) \ge \frac{\eps\lambda_{v'}}2 > 0.
\]
Hence, $0_{v'}\in \mE_+$, thus proving claim~\ref{i:edges}.
\end{proof}

\subsection{Proof of Proposition \ref{p:GW}}\label{s:proof-GW}
In the construction of the crystallisation tree started at a tree-good vertex $v$, in order to determine the number of offspring of $v$ we need only look at $\bs{U}_v$ (to determine the number of $P$-children) and $\bs{U}_{v1},\dots, \bs{U}_{vn}$ and the times $\bs{T}_{vij}$ where $i,j \in \{1,\dots, n\}$ (to determine which children of $v$ are tree-good). If a $P$-child $vi_v$ is indeed tree-good for some $i_v\in \{1,\dots, n\}$, then to determine its number of offspring we need only look at $\bs{U}_{vi_v}$ again (to determine the number of $P$-children) and $\bs{U}_{vi_v1},\dots, \bs{U}_{vi_vn}$ and the times $\bs{T}_{vi_vjk}$ where $j,k \in \{1,\dots, n\}$ (to determine which children of $v$ are tree-good). All of these random variables are independent, except that we may look at each $\bs{U}_{v'}$ twice (once to determine if $v'$ is tree-good, and once to determine its $P$-children).

Thus given that $v$ is a tree-good vertex, the crystallisation tree $\Cr_v$ is a Galton Watson tree with offspring distribution given by the number of tree-good $P$-children of $v$. 

To show that the crystallisation tree can be made supercritical, we need to show that the average number of offspring can be made greater than~$1$.

The main step that we have to make now is to show that the good events appearing in Definition~\ref{d:good} occur with high probability. Recall that $V^- := V \setminus \{x_0\}$ denotes the family of non-root vertices of $\TT_n$.
\begin{lemma}[High probability of good events]
	\label{l:good}
	Let $\a > 1$, $n \ge 2$ and $v \in V^-$. Then,
	$$\liminf_{M \to \infty}\liminf_{\e \to 0}\liminf_{q \to 0} \P(E_i) =1$$
	for each of the following events $E$:
	\begin{itemize}
		\item[(I)] $E_1 = \{\text{$v$ is time-regular}\}$
		\item[(II)] $E_2 = \{\text{$v$ is time-good}\}$
		\item[(III)] $E_3 = \{\text{$v$ is P\'olya non-disturbing}\}$
		\item[(IV)] $E_4 = \{\text{$v$ is P\'olya-well behaved}\}$.
	\end{itemize}
\end{lemma}
Note that since the distribution of the scaled variables $\{\lambda_v T_{k, v}\}_{k \ge 1}$ does not depend on $v$, neither does the probability of the event $E_i$.

\begin{proof}
	We show that 
	$$\liminf_{M \to \infty}\liminf_{\e \to 0}\liminf_{q \to 0} \P(E_i) \ge 1 - \delta$$
	for every fixed $\delta > 0$.

\medskip

\noindent (I) 	Note that the probability of this event does not depend on $q$ or $M$ at all. 
	By the strong law of large numbers, there exists $k_1 \ge 1$ such that 
	$$\P\big(\lambda_v T_{k, v} \in (k/2, 2k) \text{ for all $k \ge k_1$} \big) \ge 1 - \delta/2.$$
	Fixing such a $k_1$, we then choose {$\eps\in (0,1/2)$} sufficiently small such that 
	$$\P\big(\lambda_v T_{k, v} \in (k\eps, k\eps^{-1}) \text{ for all $k \le k_1$} \big) \ge 1 - \delta/2.$$
	
	\medskip
	
\noindent (II) Follows immediately from part (I) since the number of children of any particle is fixed.	
	
	\medskip

\noindent 	(III) Note that $(a+b)^\a\le 2^\a (a^\a+b^\a)$ for any $a,b,\a>0$. Applying this to the condition in~\eqref{eq:non-dist} we obtain for every $j$
	\begin{align*}
		\left(\frac{2m + \eps^{-2}q j}{j/2} \right)^\a
		\le {8^\a m^\a}{j^{-\a}} + 4^\a\e^{-2\a}q^\a.
	\end{align*}
	Hence, by the union bound, 
	\[
	 1 - \P(E_3) \le 
	 	\sum_{j=M+1}^{M'} n \left(\frac{2m + \eps^{-2}q j}{j/2} \right)^\a \le
	 	8^\a m^\a n \sum_{j > M}j^{-\a} + n 4^\a\e^{-2\a} \cdot \underbrace{\lceil \eps^{-2} q^{-1} M\rceil}_{M'} \cdot q^\a.
	 \]
	Since $\a > 1$, for all $M$ sufficiently large the first term on the right hand side is smaller than $\delta/2$. On the other hand, for any $M,\vep$, the second term on the right hand side becomes less than $\delta/2$ once $q$ is sufficiently small (it scales as $q^{\a-1}$).
	
	\medskip
	
\noindent (IV) 
Recall that $\delta_0$ is fixed (and does not depend on $M$, $\eps$, $q$). By the strong law of large numbers, there exists $j_1 \ge 1$ such that 
\begin{equation}
\label{eq:j1}
	\P\Big( \sup_{j- M' >  j_1} \frac 1j\#\Big\{i\in \{M' + 1,\dots, j \}:\, U_{i,v} \ge 1 - \frac{\delta_0}2\Big\} \le \delta_0\Big)\ge 1 - \delta
\end{equation}
	holds uniformly over $M$, $\eps$ and $q$. Fix such a $j_1$.
	
	 Now, choose $q$ sufficiently small such that $j_1/M'\le \delta_0$. Then, for $j\le M'+j_1$ the inequality in~\eqref{eq:well} holds automatically, as 
	$$	 \sup_{ 1 \le j - M' \le  j_1} \frac 1j\#\Big\{i\in \{M' + 1,\dots, j \}:\, U_{i,v} \ge 1 - \frac{\delta_0}2\Big\} \le \frac {j_1}{M'}\le \delta_0;$$
	joining this with~\eqref{eq:j1} (that handles $j> M'+j_1$), we conclude the proof.
\end{proof}

Recall that $\mc T_v$ denotes the event that $v$ is tree-good.

{
\begin{corollary}
\label{cor:P2}
For $\a, n,m$ as in Theorem \ref{t:Polya},
$$\liminf_{M \to \infty}\liminf_{\e \to 0}\liminf_{q \to 0} \P(\mc T_v)> 4/5.$$
In particular, there exist sufficiently small $q$, sufficiently small $\eps$ and sufficiently large $M$, so that the crystallisation tree is supercritical. 
\end{corollary}
\begin{proof}
The first claim follows from the definition of tree good, Lemma \ref{l:good} and Corollary \ref{cor:Polya} (the latter shows that the probability that a vertex $u$ is P\'olya-regular is at least $4/5$). The number of offspring of such a vertex dominates a Binomial$(5,4/5)$ random variable (since 5 is the minimal number of $P$-children of a tree-good vertex, and each child has probability at least $4/5$ of being tree-good. Thus, the expected number of children of a vertex $u$ in the crystallisation tree is at least $4=5 \cdot 4/5$.
\end{proof}
}

\section{Single P\'olya urn: proof of Theorem~\ref{t:Polya}}\label{s:Polya}

Let $\a>1$ be fixed. As stated in the introduction, for the proof of Theorem~\ref{t:Polya} we will use Rubin's construction~\cite{davis}. Namely, we introduce families $(\xi_j,\eta_{j,k})_{j \in \N, k = 1,\dots ,n}$ of i.i.d.~exponential random variables with mean 1, and consider partial sums $S^{(r)}$ and $Z_k^{(r)}$ (see~\eqref{partials})
of the series 
$$
S= \sum_{j = m}^{\infty}\frac{\xi_j}{j^\a} \quad \text{ and } \quad Z_k = \sum_{j = 1}^{\infty}\frac{\eta_{j,k}}{j^\a}, \quad k=1,\dots, n
$$
respectively. As noted in the earlier sketch proof, these sums can be coupled with the P\'olya urn: one puts balls of the corresponding colours at these ``times''.

The conclusion of Theorem~\ref{t:Polya} (for a particular realisation) will then be satisfied if
\begin{itemize}
\item There exist at least $5$ (random) colours $k=k_1,\dots,k_5$ such that $Z_k^{(m-1)}<S<Z_k^{(m)}$.
\item $Z_k^{(2m)}>S$ for all~$k$.
\end{itemize}

Let us show that we can choose the number $n$ of colours so that both of the above conditions are fulfilled with probability greater than~$\frac45$. We will first deduce it from a chain of auxiliary statements, postponing their proofs, and will subsequently prove these auxiliary statements.

\subsection{Deducing the theorem from auxiliary statements}

The first of these statements is devoted to the description of the random sum $S=\sum_{j = m}^{\infty}\frac{\xi_j}{j^\a}$. As it has many ($\sim m$) terms of roughly the same magnitude (in the spirit of the Law of Large Numbers and Central Limit Theorem) its sum is likely to be close to its expectation. This expectation is close to $\frac1{(\a-1)m^{\a-1}}$ due to a comparison with the integral $\int_m^{\infty}x^{-\a} {dx}$. More precisely, we have the following.

\begin{proposition}\label{p:S}
There exists a constant $C_1 > 0$ such that for all sufficiently large $m$ one has 
		\begin{equation}\label{eq:S-Cheb}
		\P\left(\Big|S-\frac{1/(\a-1)}{m^{\a-1}}\Big| < \frac{C_1}{m^{\a-\frac12}}\right) \ge 1-\frac1{100}.
		\end{equation}
\end{proposition}

From now on we fix the constant $C_1$ given by Proposition~\ref{p:S}, and let $s_-,s_+$ be the endpoints of the confidence interval of values of~$S$:
\begin{equation}\label{eq:spm}
s_{\pm}: = \frac1{(\a-1)m^{\a-1}} \pm \frac{C_1}{m^{\a-\frac12}} = \left( \frac1{\a-1} \pm \frac{C_1}{\sqrt m} \right) \cdot \frac1{m^{\a-1}}.
\end{equation}
Then, Proposition~\ref{p:S} states that $\P(S\notin (s_-,s_+))<1/100$. 

Consider now the sums
$$
Z'_k:= Z_k^{(m-1)}= \sum_{j = 1}^{m-1}\frac{\eta_{j,k}}{j^\a }, \quad Z''_k:= Z_k^{(2m)}-Z'_k = \sum_{j = m}^{2m}\frac{\eta_{j,k}}{j^\a }.
$$
{
As the behaviour of the sums $Z_k^{(r)}$ is independent from that of $S$, it suffices to establish that the following conclusions hold with sufficiently high probability for any $s\in [s_-, s_+]$ (see Fig.~\ref{f:geometry})}
\begin{enumerate}
\item\label{c:five} There exist at least $5$ (random) colours $k=k_1,\dots,k_5$ such that $Z'_k<s<Z'_k+\frac{\eta_{m,k}}{m^\a }$.
\item\label{c:all} $Z'_k+Z''_k>s$ for all~$k$.
\end{enumerate}
For a given value $s$ we call $k$ such that $Z'_k<s$ the \emph{candidate colours}.

Note that these conditions, in a sense, disregard indices $k$ such that $Z'_k>s$: condition~\ref{c:all} is satisfied automatically, but such $k$ cannot be one of those appearing in condition~\ref{c:five}. We thus fix $n$ in such a way that even for the sum $S$ taking the smallest possible non-discarded value~$s_-$, we still have enough candidate colours $k$ for the first condition. Namely, we fix
\begin{equation}\label{eq:N}
n:= \Big\lceil \frac{100 m^\a }{\P(Z'_k<s_-)} \Big\rceil.
\end{equation}

\begin{figure}[!h!]
\includegraphics{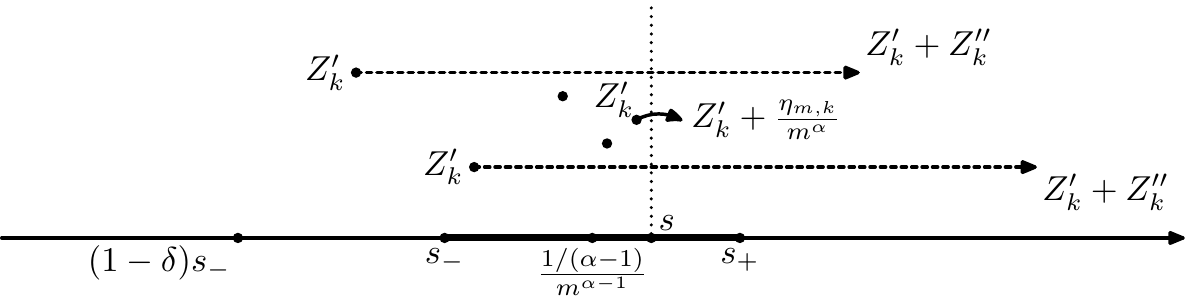}
\caption{Checking claims~\ref{c:five} and~\ref{c:all} conditionally on $\{S=s\}$, where $s\in [s_-,s_+]$. Bold dots represent candidate colours. They all belong to $[(1-\delta)s_-,s]$; we have to check $Z'_k+Z''_k>s$ for all of them (dashed horizontal arrows, ensuring claim~\ref{c:all}), and find at least five of them with $Z'_k+\frac{\eta_{m,k}}{m^\a}>s$ (curved arrow, claim~\ref{c:five}).}\label{f:geometry}
\end{figure}

The following will be used to show that conclusion~\ref{c:five} holds with probability close to 1. 
\begin{proposition}\label{p:Z-incr}
The density of the random variable $Z'_k$ is increasing on $[0,s_+]$ for all sufficiently large~$m$. In particular, for every $s\in [s_-,s_+]$ one has 
\begin{equation}\label{eq:Zp}
n\cdot \P\Big(s-\frac1{m^\a}<Z'_k<s\Big) >100.
\end{equation}
\end{proposition}

Let us now ensure that the condition~\ref{c:all} is satisfied with high probability. On the one hand, the added sum $Z''_k=\sum_{j=m}^{2m}\frac{\eta_{j,k}}{j^\a }$ also consists of many comparable summands, and hence admits a Large Deviation Theorem-type lower bound. 
\begin{proposition}\label{p:LargeDev}
There exists a constant $c_2>0$ such that for all sufficiently large $m$,
\begin{equation}\label{eq:Zpp}
\P\Big(Z''_k > \frac1{10\cdot (2m)^{\a-1}}\Big)\ge 1-e^{-c_2 m}.
 \end{equation}
\end{proposition}

On the other hand, the values of $S$ are (for large $m$) quite small, while the random variable $Z'_k$ is a sum of many independent random variables, and its density (or the partition function) should decrease quite fast near~$0$; we give a rigorous treatment of this informal argument below (see Definition~\ref{def:growth} below). In particular, it is natural to expect all of the estimated $\sim 100 m^\a $ values of $Z'_k$ that fall on $[0,s_-]$ to be concentrated near the right end of this interval. This is guaranteed by the following.

\begin{proposition}\label{p:delta}
For every $\delta>0$,
\[
\lim_{m\to\infty} \frac{m^\a  \P(Z'_k<(1-\delta) s_-)}{\P(Z'_k<s_-)} =0.
\]
\end{proposition}

Finally, we have that the expected number of candidate colours grows sub-exponentially with $m$. This will be combined with the previous two propositions to establish (ii) with high probability.
\begin{proposition}\label{p:subexp}
There is a constant {$C_2>0$} such that for all $m$ sufficiently large 
one has 
\[
n \cdot \P(Z'_k<{s_+}) <e^{C_2\sqrt m}. 
\]
\end{proposition}

\begin{proof}[Proof of Theorem \ref{t:Polya}]
Fix $\a>1$. Let $J_m$ be the event that only colour 0 is selected more than $2m-1$ times and there are at least 5 colours selected exactly $m-1$ times. We need to show that for sufficiently large $m$ (and with $n$ as in \eqref{eq:N}), $\P(J_m)>4/5$.

For $s\in [s_-,s_+]$, let 
\[
K_{m,s}=\#\{k\in 1,\dots, n:Z'_k<s<Z'_k+\frac{\eta_{m,k}}{m^\a }\},
\] and 
\[
I_{m,s}=\cap_{k=1}^n\{Z'_k+Z''_k>s\}.
\] 
Then, $\{K_{m,s}\ge 5\}$ and $I_{m,s}$ are the events that correspond to claims~\ref{c:five} and~\ref{c:all} respectively for a given value of~$s$; denote also the even that both these claims (for a given $s$) are fulfilled by
\begin{align}
J_{m,s}:=\{K_{m,s}\ge 5\} \cap I_{m,s}.
\end{align}
We will show that for all $m$ sufficiently large, 
\begin{equation}
\label{Jclaim}
\inf_{s \in [s_-,s_+]}\P(J_{m,s})\ge 88/100.
\end{equation}

To do so, first note that
\[R_k:=\{Z'_k<s<Z'_k+\frac1{m^\a }\}\cap \{\eta_{m,k}>1\}\subset \{Z'_k<s<Z'_k+\frac{\eta_{m,k}}{m^\a }\}.\]
By Proposition \ref{p:Z-incr} we have that $\P(Z'_k<s<Z'_k+\frac1{m^\a })>100/n$ and since $\eta_{m,k}$ is independent of $Z'_k$ and $\P(\eta_{m,k}>1)=e^{-1}$ we have $\P(R_k)\ge 100/(en)\ge 30/n$. 

Now, let $X$ be the number of $k$ for which $R_k$ occurs. Then $X$ is a binomial random variable with  expectation $\E[X]=n \cdot\P(R_k) \ge 30$.  Now, 
\[\P(X<5)\le \P(|X-\E[X]|>\E[X]-5),\] 
which by Chebyshev's inequality is at most 
\[
\frac{\Var(X)}{(\E[X]-5)^2} \le 
\frac{n\P(R_k)}{n\P(R_k)-5}\cdot \frac1{n \P(R_k)-5}\le \frac{30}{25^2}<\frac1{10}.
\] 
Thus $\P(X\ge 5)\ge 9/10$ and since $K_{m,s}\ge X$ we have $\P(K_{m,s}\ge 5)\ge 9/10$ for all $s\in [s_-,s_+]$.

Turning our attention to the event $I_{m,s}$, let $\Col=\{k\le n:Z'_k<s\}$. Then, the event $I_{m,s}$ (up to sets of measure 0) is equal to the event 
\[
I'_{m,s}:=\{Z_k'+Z_k''>s \text{ for every }k \in \Col\}
\] 
(because if $k \notin \Col$ then $Z'_k\ge s$ so $Z'_k+Z''_k>s$ a.s.). Choose 
\[
\delta=\frac{\a-1}{20\cdot 2^\a },
\]
and let $s'_-=(1-\delta)s_-$. If $Z'_k>s'_-$ and $Z_k''>\frac1{10\cdot (2m)^{\a-1}}$ then 
\begin{align*}
Z'_k+Z''_k &> 
 (s_- - \delta s_-) + \frac1{10\cdot (2m)^{\a-1}}
 \\
&> \frac1{m^{\a-1}}\cdot \left(\frac1{\a-1} - \frac{C_1}{\sqrt m} - \frac{\delta}{\a-1} + \frac1{10\cdot 2^{\a-1}}\right)\\
&=\frac1{m^{\a-1}}\cdot \left(\frac1{\a-1} - \frac{C_1}{\sqrt m} + \frac1{20\cdot 2^{\a-1}}\right).
\end{align*}
For $m$ sufficiently large we have 
$\frac1{20\cdot 2^{\a-1}}>2\frac{C_1}{\sqrt m}$, and hence for all such $m$, 
\[
Z'_k + Z''_k > \frac1{m^{\a-1}}\cdot \left(\frac1{\a-1} + \frac{C_1}{\sqrt m} \right) = s_+ \ge s,
\]
so 
\begin{align*}
I'_{m,s}&\supset \Big\{\text{$\forall k\in \Col$ we have }Z'_k>s'_- \text{ and }Z_k''>\frac1{10\cdot (2m)^{\a-1}}\Big\}\\
&\supset \Big\{Z'_k>s'_-\quad \text{ $\forall k \in [n]$ }\Big\}\cap \Big\{ Z_k''>\frac1{10\cdot (2m)^{\a-1}}\quad \text{ $\forall k\in \Col$}\Big\}.
\end{align*}	
Thus,
\begin{align}
(I'_{m,s})^c&\subset \Big(\cup_{k=1}^n \{Z'_k<s'_-\}\Big)\cup \Big\{ Z_k''\le \frac1{10\cdot (2m)^{\a-1}}\quad \text{for some $ k\in \Col$}\Big\}.\label{ouch1}
\end{align}
The probability of the first event on the right hand side of \eqref{ouch1} is at most $n\cdot \P(Z'_k<s'_-)$ which by the choice of $n$ satisfies
\begin{align*}
n\cdot \P(Z'_k<s'_-)&\le \frac{101 m^\a}{\P(Z'_k<s_-)} \cdot\P(Z'_k<s'_-)
\\
&=101 \, \frac{m^\a\P(Z'_k<(1-\delta) s_-)}{\P(Z'_k<s_-)}.
\end{align*}
By Proposition \ref{p:delta} the product in the right hand side is less than $1/100$ for all $m$ sufficiently large. 

Now we turn to the probability of the second event on the right hand side of \eqref{ouch1}. It is equal to 
\begin{align}
\P\Big( Z_k''\le \frac1{10\cdot (2m)^{\a-1}}\quad \text{for some $ k\in \Col$}\Big)
&
\le \E \Bigg[ \# \left\{ k \in \Col \text{ such that }  Z_k''\le \frac1{10\cdot (2m)^{\a-1}} \right\} \Bigg]\\
&=\P\Big(Z_1''\le \frac1{10\cdot (2m)^{\a-1}}\Big)\E[\#\Col],\label{ouch2}
\end{align}
where in the last step we have used the fact that the $Z''_k$ are independent of $\Col$. Note that $\#\Col$ is dominated by a Bin$(n,\P(Z'_k<s_+))$ random variable, so by Propositions \ref{p:subexp} and \ref{p:LargeDev}, the value in~\eqref{ouch2} is at most
\[e^{-c_2 m}\cdot e^{C_2\sqrt m} = o(1) \quad \text{ as } m\to\infty.\]
In particular, this probability becomes less than $1/{100}$ for all sufficiently large~$m$, so for all such $m$, $\P(I_{m,s})=\P(I'_{m,s})\ge 98/100$ for all $s\in[s_-,s_+]$.

We have shown that for all $m$ sufficiently large (and $n$ defined by \eqref{eq:N}), for all $s \in [s_-,s_+]$, 
\[\P(J_{m,s}^c)\le  \P(\{K_{m,s}< 5\}) + \P(I_{m,s}^c) \le \frac1{10}+\frac{2}{100}.\]
Therefore the claim \eqref{Jclaim} holds. To complete the proof, we simply condition on $S$. Note that 
\begin{align*}
\P(J_m)&\ge \P\big(J_m|S\in (s_-,s_+)\big)\cdot \P\big(S \in (s_-,s_+)\big)\\
&\ge \inf_{s \in [s_-,s_+]}\P(J_{m,s})\cdot \P\big(S \in (s_-,s_+)\big)\\
&\ge \frac{88}{100} \cdot \frac{99}{100}>\frac45,
\end{align*}
where the penultimate bound uses Proposition \ref{p:S}, and is valid for all $m$ sufficiently large.
\end{proof}

\subsection{Proofs of auxiliary statements}

\begin{proof}[Proof of Proposition~\ref{p:S}]

		Let us estimate the expectation and the variance of the variable $S$:
		\[
		\E [S] = \sum_{j = m}^{\infty} \frac1{j^\a}.
		\]
		Comparing to the integral gives 
		\[
			\sum_{j = m}^{\infty} \frac1{j^\a} > \int_m^{\infty} x^{-\a}\d x = \frac1{(\a-1)m^{\a-1}} > \sum_{j = m+1}^{\infty} \frac1{j^\a},
		\]
		hence 
		\[
		\left| \E [S ]- \frac1{(\a-1)m^{\a-1}} \right| \le \frac1{m^\a }.
		\]
		In the same way, 
		\[
		\Var(S) =  \sum_{j = m}^{\infty} \frac1{j^{2\a}}, 
		\]
		and thus 
		\[
		\frac1{(2\a-1)m^{2\a-1}} < \Var(S) < \frac1{(2\a-1)m^{2\a-1}} + \frac1{m^{2\a}}.
		\]
		This gives a standard deviation of the type $\const \cdot m^{-(\a-\frac12)}$. Applying Chebyshev's inequality, we see that there exists a constant $C_1$ such that for all sufficiently large $m$ one has 
		\begin{equation}
		\P\left(\Big|S-\frac1{(\a-1)m^{\a-1}}\Big| < \frac{C_1}{m^{\a-\frac12}}\right) \ge 1-\frac1{100}.
		\end{equation}
\end{proof}

\begin{proof}[Proof of Proposition~\ref{p:Z-incr}]
First, the density of the sum $\frac{\eta_{k,1}}{1^\a }+\frac{\eta_{k,2}}{2^\a }$ can be computed explicitly as
		\[\rho(t) = \frac{2^\a}{2^\a-1}\left[e^{-t}-e^{-2^\a t}\right]\indic{t>0}.
		\]
		The derivative is positive for $t<(2^\a-1)^{-1}\log 2^\a$, so if $m$ is sufficiently large then $\rho$ is increasing on $[0,s_+]$.
		
		Next, if a (density) function is supported on~$\R_+$ and is increasing on $[0,s_+]$, its convolution with any density function supported on the positive numbers is still increasing on $[0,s_+]$. This proves the first claim of the proposition. It also implies that for $0\le a<b\le s_+$ we have 
		\[(b-a)^{-1}\P(a<Z'_k<b)\ge b^{-1}\P(Z'_k<b),
		\]
	whence for any $s\in [s_-,s_+]$ we have 
\begin{align*}
	n\cdot \P\Big(s-\frac1{m^\a}<Z'_k<s\Big) &\ge n\cdot \P\Big(s_- -\frac1{m^\a}<Z'_k<s_-\Big)\\
	 & \ge 
 n\cdot  \P(Z'_k<s_-) \cdot \frac1{m^\a s_-}\\
 & \ge 100 m^\a  \cdot \frac1{m^\a s_-}
\end{align*}
which is greater than 100 for $m$ sufficiently large so that $s_-<1$.
\end{proof}

\begin{proof}[Proof of Proposition~\ref{p:LargeDev}]
The sum $Z_k''$ can be estimated from below as
\[
Z_k'' =\sum_{j = m}^{{2m}} \frac{\eta_{j,k}}{j^\a} > \frac1{(2m)^\a } \sum_{j=m}^{{2m}} \eta_{j,k} \ge 
\frac1{2\cdot (2m)^{\a-1}} \cdot \frac1{m}\sum_{j=m+1}^{2m} \eta_{j,k}.
\]
The statement of the proposition then follows directly from standard large deviations estimates (Chernoff's bound)
applied to the i.i.d. random variables $\eta_{j,k} \sim \mathrm{Exp}(1)$, since the probability of the event
$$
\Big\{\frac1{m}\sum_{j=m + 1}^{2m} \eta_{j,k} \le \frac15\Big\}
$$
is exponentially small in~$m$.
\end{proof}
In order to prove Proposition \ref{p:subexp} we will use the following lemma.
\begin{lemma}
		\label{lem:s's}
		For any $s'>s>0$ we have 
		$$
		\P(Z'_k<s')<\left(\frac{s'}{s} \right)^{m} \P(Z'_k<s).
		$$
	\end{lemma}
	\begin{proof}
		Let $\rho_j (x) := j^\a  \exp ( -j^\a  x)\indic{x>0}$ denote the density of $j^{-\a} \eta_{j,k}$. Then
		\begin{equation}\label{eq:P-number}
		\P(Z'_k<s') = \int_{x_1+\dots+x_{m-1}<s'} \prod_{i=1}^{m-1}\rho_i(x_i)\, \d x_1 \cdots \d x_{m-1}.
		\end{equation}
		Now, let us make a change of variables in~\eqref{eq:P-number}, denoting $y_j = \frac{s}{s'}x_j$. We get 
		\[
			\left(\frac{s'}{s} \right)^{m-1} \int_{y_1+\cdots+y_{m - 1}<s} \prod_{i=1}^{m-1} \rho_i\big(\frac{s'}{s} y_i\big) \, \d y_1 \cdots \d y_{m-1},
		\]
		and as all the densities $\rho_j$ are monotone decreasing, and $ \frac{s'}{s} y_j>y_j$, this is less than
		\[
		\left(\frac{s'}s \right)^{m-1} \int_{y_1+\cdots+y_m<s} \prod_{i=1}^{m-1} \rho_i(y_i) \, \d y_1 \cdots \d y_{m-1} = \left(\frac{s'}{s} \right)^{m-1} \P(Z'_k<s)\le \left(\frac{s'}s \right)^m \P(Z'_k<s).
		\]
	\end{proof}

\begin{proof}[Proof of Proposition~\ref{p:subexp}]
We apply Lemma~\ref{lem:s's} with $s = s_-$ and $s' = s_+$ to obtain 
		$$
			\P(Z'_k<s_+)<\left(\frac{s_+}{s_-} \right)^m \P(Z'_k<s_-).
		$$ 
		Now, 
		\[
			\frac{s_+}{s_-} = \frac{1+\frac{C_1{(\a-1)}}{\sqrt m}}{1-\frac{C_1{(\a-1)}}{\sqrt m}} = \Big(1+\frac{2C_1(\a-1)}{\sqrt m} + o(\frac1{\sqrt m})\Big).
		\]
		As $n\P(Z'_k<s_-) \sim 100 m^\a $ due to~\eqref{eq:N}, we have 
		\[
			n \P(Z'_k<s_+) < 100 m^\a  (1+o(1)) \cdot \left( 1+\frac{2C_1(\a-1)}{\sqrt m} + o(\frac1{\sqrt m}) \right)^m < e^{(2C_1(\a-1)+1) \sqrt m}
		\]
		for all $m$ large enough.
\end{proof}

To prove Proposition \ref{p:delta}, we will need an estimate in the other direction than for the proof of Proposition~\ref{p:subexp}. If the estimates of Lemma~\ref{lem:s's} are saying that the partition function grows \emph{at most} as $m$th power of the argument, for the proof of this proposition, we will need a \emph{lower} estimate, though by some smaller power. Namely, let us introduce the following definition.
	\begin{definition}\label{def:growth}
		For $s>0$, a non-atomic positive random variable $\xi$ is \emph{$p$-growing up to $s$} if for any $0<x<y\le s$ one has 
		\begin{equation}\label{eq:def-growth}
		\P(\xi<x)\le e \left(\frac xy \right)^p \P(\xi<y).
		\end{equation}
	\end{definition}

Next, note that establishing such a growth for a sum of a part of the summands in a large sum implies the same growth for the total sum: 
	\begin{lemma}
		\label{lem:pgrow}
		Let $\xi,\eta$ be independent random variables, where $\xi$ is $p$-growing up to $s$, and $\eta$ is positive. 
		Then, $\xi+\eta$ is also $p$-growing up to $s$.
	\end{lemma}
	\begin{proof}
		First, for any $a>0$ the shifted random variable $\xi+a$ is $p$-growing up to $s$: for $x\le a$ the inequality~\eqref{eq:def-growth} holds automatically, while for $x>a$ it follows immediately from the inequality $\frac{x-a}{y-a}< \frac xy$.  Thus, 
			$$\P(\xi+\eta<x \mid \eta)\le e \left(\frac xy\right)^p\P(\xi+\eta<y \mid \eta)$$ almost surely, and taking expectations gives the result.
	\end{proof}

Our final lemma finally establishes the desired growth:
\begin{lemma}\label{l:m-beta}
There exists $\beta>0$ such that for all $m$ sufficiently large $Z'_k$ is $m^{\beta}$-growing on $[0,s_-]$.
\end{lemma}
\begin{proof}
Actually, we will show that the conclusion of the lemma is satisfied (for all $m$ large enough) for any $\beta<\frac{\a-1}\a $. Namely, let $p=\lceil m^{\beta} \rceil$ and consider the sum $Z_k^{(p)}$ of the first $p$ summands in the definition of~$Z'_k$. We are going to show that there exists a constant $C_p$ such that one has 
\begin{equation}\label{eq:x-p}
\forall x\in [0,s_-] \qquad \frac1{e} \, C_p x^p\le \P(Z_k^{(p)}<x) \le C_p x^p.
\end{equation}
It is easy to see that~\eqref{eq:x-p} implies that $Z_k^{(p)}$ is $p$-growing on $[0,s_-]$, and due to Lemma~\ref{lem:pgrow} this will imply the desired $p$-growth of $Z'_k = Z_k^{(p)} + \eta$, where $\eta$ is the sum of the last $m-1-p$ summands in the definition of $Z'_k$.

Now, the probability $\P(Z_k^{(p)}<x)$ is given (cf.~\eqref{eq:P-number}) by an integral 
		\begin{equation}\label{eq:P-number-new}
		\P(Z_k^{(p)}<x) = \int_{x_1+\cdots+x_p<x} \prod_{i=1}^p\rho_i(x_i) \, \d x_1 \cdots \d x_p,
		\end{equation}
where $\rho_j(x_j)=j^\a \exp ( -j^\a x_j)\indic{x_j>0}$. The volume of the domain of integration is equal to $\frac{x^p}{p!}$, and the function under the integral is equal to 
$$
\rho_1(x_1)\cdots \rho_p(x_p) = (p!)^\a \cdot \exp (-\sum_{j=1}^p j^\a  x_j).
$$
Hence, 
\begin{equation}\label{eq:Z-p!}
 \frac{x^p}{(p!)^{1-\a}}\cdot \exp (-p^\a  x) < \P(Z_k^{(p)}<x) < \frac{x^p}{(p!)^{1-\a}},
\end{equation}
where we have used the inequalities $\sum_{j=1}^p j^\a  x_j \le p^\a  \sum_{j=1}^p x_j \le p^\a  x$ that hold on the domain of integration in~\eqref{eq:P-number-new}. Now, take $C_p:=\frac1{(p!)^{1-\a}}$; then~\eqref{eq:Z-p!} implies the desired~\eqref{eq:x-p}, once $p^\a  s_-<1$. As $s_-<\frac1{(\a-1)m^{\a-1}}$, and $p\sim m^{\beta}$, this inequality holds for all sufficiently large $m$ if $\a \beta< \a-1$. The conclusion of the lemma holds (for all sufficiently large $m$) for any $\beta<\frac{\a-1}\a $, thus concluding the proof.
\end{proof}

\begin{proof}[Proof of Proposition~\ref{p:delta}] 
Fix $\beta>0$ as in Lemma~\ref{l:m-beta} taking $x=(1-\delta)s_-$ and $y=s_-$ in~\eqref{eq:def-growth} then provides the upper bound 
	\begin{equation}\label{eq:m-beta}
		\frac{m^\a \P(Z'_k<(1-\delta)s_-)}{\P(Z'_k<s_-)}< m^\a e \left(1-\delta \right)^{m^\beta},
	\end{equation}
	and for any fixed $\delta>0$ the right hand side of~\eqref{eq:m-beta} converges to $0$ as $m\to\infty$.
\end{proof}

\section*{Acknowledgements}
The authors thank an anonymous referee of the paper \cite{compass} for asking a question that in part motivated this work. CH's research is supported by the Centre
for Stochastic Geometry and Advanced Bioimaging, funded by grant 8721 from the Villum Foundation.
MH's research is supported by Future Fellowship FT160100166 from the Australian Research Council. 
VK's research is partially supported by the project ANR Gromeov (ANR-19-CE40-0007), as well as by by the Laboratory of Dynamical Systems and Applications NRU HSE, of the Ministry of science and higher education of the RF grant ag. No. 075-15-2019-1931.

\bibliographystyle{plain}

\end{document}

%% file: outline.tex
\begin{tikzpicture}


	\draw[->] (1,-3.5)--(14.5,-3.5);

	\coordinate[label={[align=center]:\scriptsize{$v$ doesn't} \\ \scriptsize{fire}}] (A) at (2.0, -3.50);
	\fill (3,-3.5) circle (2pt);
	\coordinate[label={[align=center]below:\scriptsize{$t_0(v)$} }] (A) at (3.0, -3.50);

	\coordinate[label={[align=center]:\scriptsize{$v$ fires $\ge M$} \\ \scriptsize{times}}] (A) at (4.3, -3.50);
	\fill (5.2,-3.5) circle (2pt);
	\coordinate[label={[align=center]below:\scriptsize{$t_1(v)$} }] (A) at (5.2, -3.50);

	\coordinate[label={[align=center]:\scriptsize{$v$ fires only towards $\bar v$}}] (A) at (6.7, -3.50);
	\fill (8.3,-3.5) circle (2pt);
	\coordinate[label={[align=center]below:\scriptsize{$t_0(v') = \tfrac{t_0(v)}q$} }] (A) at (8.3, -3.50);

	\coordinate[label={[align=center]:\scriptsize{$v'$ fires $\ge M$} \\ \scriptsize{times}}] (A) at (9.4, -3.50);
	\fill (10.4,-3.5) circle (2pt);
	\coordinate[label={[align=center]below:\scriptsize{$t_1(v')$} }] (A) at (10.4, -3.50);

	\coordinate[label={[align=center]:\scriptsize{$v'$ fires only towards $v$}}] (A) at (12.0, -3.50);
	\fill (13.5,-3.5) circle (2pt);
	\coordinate[label={[align=center]below:\scriptsize{$t_0(v'') = \tfrac{t_0(v')}q$} }] (A) at (13.5, -3.50);

	\begin{scope}[shift={(0,0), scale=0.69}]
\node[circle,fill = black,scale = 0.3] (A) at (1,0) {};
\node[circle,fill = black,scale = 0.3] (v) at (2,0) {};
\node[circle,fill = black,scale = 0.3] (B) at (4,2) {};
\node[circle,fill = black,scale = 0.3] (C) at (4,0.5) {};
\node[circle,fill = black,scale = 0.3] (E) at (4,-0.5) {};
\node[circle,fill = black,scale = 0.3] (F) at (4,-2) {};

\node[circle,fill = black,scale = 0.3] (B1) at (5,2.6) {};
\node at (4.0,2.2) {${v'}$};
\node at (4.8,2.6) {$\sss1$};
\node at (4.8,2.25) {$\sss1$};
\node at (4.8,1.95) {$\sss1$};
\node at (4.8,1.6) {$\sss1$};
\node at (4.6,2.25) {\scalebox{.5}{$\vdots$}};
\node at (4.6,1.8) {\scalebox{.5}{$\vdots$}};
\node at (4.4,0.5) {$\dots$};
\node at (4.4,0.4) {$\ddots$};
\node at (4.4,0.8) {$\iddots$};

\node at (4.4,0.5) {$\dots$};
\node at (4.4,0.4) {$\ddots$};
\node at (4.4,0.8) {$\iddots$};

\node at (4.4,-0.5) {$\dots$};
\node at (4.4,-0.6) {$\ddots$};
\node at (4.4,-0.2) {$\iddots$};

\node at (4.4,-2) {$\dots$};
\node at (4.4,-2.1) {$\ddots$};
\node at (4.4,-1.7) {$\iddots$};

\node[circle,fill = black,scale = 0.3] (B3) at (5,2.2) {};
\node[circle,fill = black,scale = 0.3] (B5) at (5,1.8) {};
\node[circle,fill = black,scale = 0.3] (B5) at (5,1.4) {};
\draw (4,2)--(5,2.6);
\draw (4,2)--(5,2.2);
\draw (4,2)--(5,1.8);
\draw (4,2)--(5,1.4);

\draw[very thick] (1,0)--(2,0);
\draw (2,0)--(4,2);
\node at (3,0.7) {$\vdots$};
\node at (3,-0.5) {$\vdots$};
\draw (2,0)--(4,0.5);
\draw (2,0)--(4,-0.5);
\draw (2,0)--(4,-2);
\node at (1.4,0.15) {$m$};
\node at (2,-0.25) {$v$};
\node at (1,-0.25) {$\bar v$};
\node at (3.3,1.5) {$1$};
\node at (3.3,0.5) {$1$};
\node at (3.3,-0.15) {$1$};
\node at (3.3,-1.1) {$1$};
		\draw[decorate, decoration = {brace, mirror}] (5,-2.5)--(1,-2.5);
\end{scope}
	\begin{scope}[shift={(4.2,0), scale=0.69}]
\node[circle,fill = black,scale = 0.3] (A) at (1,0) {};
\node[circle,fill = black,scale = 0.3] (v) at (2,0) {};
\node[circle,fill = black,scale = 0.3] (B) at (4,2) {};
\node[circle,fill = black,scale = 0.3] (C) at (4,0.5) {};
\node[circle,fill = black,scale = 0.3] (E) at (4,-0.5) {};
\node[circle,fill = black,scale = 0.3] (F) at (4,-2) {};

\node[circle,fill = black,scale = 0.3] (B1) at (5,2.6) {};
\node at (4.0,2.2) {${v'}$};
\node at (4.8,2.6) {$\sss1$};
\node at (4.8,2.25) {$\sss1$};
\node at (4.8,1.95) {$\sss1$};
\node at (4.8,1.6) {$\sss1$};
\node at (4.6,2.25) {\scalebox{.5}{$\vdots$}};
\node at (4.6,1.8) {\scalebox{.5}{$\vdots$}};
\node at (4.4,0.5) {$\dots$};
\node at (4.4,0.4) {$\ddots$};
\node at (4.4,0.8) {$\iddots$};

\node at (4.4,0.5) {$\dots$};
\node at (4.4,0.4) {$\ddots$};
\node at (4.4,0.8) {$\iddots$};

\node at (4.4,-0.5) {$\dots$};
\node at (4.4,-0.6) {$\ddots$};
\node at (4.4,-0.2) {$\iddots$};

\node at (4.4,-2) {$\dots$};
\node at (4.4,-2.1) {$\ddots$};
\node at (4.4,-1.7) {$\iddots$};

\node[circle,fill = black,scale = 0.3] (B3) at (5,2.2) {};
\node[circle,fill = black,scale = 0.3] (B5) at (5,1.8) {};
\node[circle,fill = black,scale = 0.3] (B5) at (5,1.4) {};
\draw (4,2)--(5,2.6);
\draw (4,2)--(5,2.2);
\draw (4,2)--(5,1.8);
\draw (4,2)--(5,1.4);

\draw[line width = 3pt] (1,0)--(2,0);
\draw[very thick] (2,0)--(4,2);
\node at (3,0.7) {$\vdots$};
\node at (3,-0.5) {$\vdots$};
\draw[very thick] (2,0)--(4,0.5);
\draw (2,0)--(4,-0.5);
\draw (2,0)--(4,-2);
		\node at (1.5,0.2) {};
\node at (2,-0.25) {$v$};
\node at (1,-0.25) {$\bar v$};
\node at (3.3,1.6) {$m$};
\node at (3.3,0.5) {$m$};
\node at (3.3,-0.15) {$?$};
\node at (3.3,-1.1) {$?$};
		\draw[decorate, decoration = {brace, mirror}] (4,-2.5)--(1,-2.5);
	\end{scope}
	\begin{scope}[shift={(8.5,0),scale=0.69}]
\node[circle,fill = black,scale = 0.55] (A) at (1,0) {};
\node[circle,fill = black,scale = 0.55] (v) at (2,0) {};
\node[circle,fill = black,scale = 0.3] (B) at (4,2) {};
\node[circle,fill = black,scale = 0.3] (C) at (4,0.5) {};
\node[circle,fill = black,scale = 0.3] (E) at (4,-0.5) {};
\node[circle,fill = black,scale = 0.3] (F) at (4,-2) {};

\node[circle,fill = black,scale = 0.3] (B1) at (5,2.6) {};
\node at (4.0,2.2) {${v'}$};
\node at (5.2,2.7) {$\sss{v''}$};
\node at (4.7,2.6) {$\sss{m}$};
\node at (4.7,2.25) {$\sss{m}$};
\node at (4.7,1.95) {$\sss{?}$};
\node at (4.7,1.6) {$\sss{?}$};
\node at (4.6,2.25) {\scalebox{.5}{$\vdots$}};
\node at (4.6,1.8) {\scalebox{.5}{$\vdots$}};
\node at (4.4,0.5) {$\dots$};
\node at (4.4,0.4) {$\ddots$};
\node at (4.4,0.8) {$\iddots$};

\node at (4.4,0.5) {$\dots$};
\node at (4.4,0.4) {$\ddots$};
\node at (4.4,0.8) {$\iddots$};

\node at (4.4,-0.5) {$\dots$};
\node at (4.4,-0.6) {$\ddots$};
\node at (4.4,-0.2) {$\iddots$};

\node at (4.4,-2) {$\dots$};
\node at (4.4,-2.1) {$\ddots$};
\node at (4.4,-1.7) {$\iddots$};

\node[circle,fill = black,scale = 0.3] (B3) at (5,2.2) {};
\node[circle,fill = black,scale = 0.3] (B5) at (5,1.8) {};
\node[circle,fill = black,scale = 0.3] (B5) at (5,1.4) {};
\draw[thick] (4,2)--(5,2.6);
\draw[thick] (4,2)--(5,2.2);
\draw (4,2)--(5,1.8);
\draw (4,2)--(5,1.4);

\draw[line width = 5pt] (1,0)--(2,0);
\draw[line width = 3pt] (2,0)--(4,2);
\node at (3,0.7) {$\vdots$};
\node at (3,-0.5) {$\vdots$};
\draw[line width = 3pt] (2,0)--(4,0.5);
\draw (2,0)--(4,-0.5);
\draw (2,0)--(4,-2);
		\node at (1.5,0.3) {};
\node at (2,-0.25) {$v$};
\node at (1,-0.25) {$\bar v$};
		\node at (3.3,1.8) {};
		\node at (3.5,0.7) {};
\node at (3.3,-0.15) {$?$};
\node at (3.3,-1.1) {$?$};
		\draw[decorate, decoration = {brace, mirror}] (5.5,-2.5)--(1.8,-2.5);
	\end{scope}

\end{tikzpicture}

%% file: neighb.tex
\begin{tikzpicture}


	\draw[->] (.5,-3.5)--(14.5,-3.5);

	\fill (3,-3.5) circle (2pt);
	\coordinate[label={[align=center]below:{$t_0(v)$} }] (A) at (3.0, -3.50);

	\fill (7.4,-3.5) circle (2pt);
	\coordinate[label={[align=center]below:{$T_{M, v}$} }] (A) at (6.9, -3.50);

	\fill (9.5,-3.5) circle (2pt);
	\coordinate[label={[align=center]below:{$t_1(v)$} }] (A) at (9.5, -3.50);

	\fill (11.4,-3.5) circle (2pt);
	\coordinate[label={[align=center]below:{$t_0(v') = \tfrac{t_0(v)}q$} }] (A) at (11.4, -3.50);

	\fill (13.5,-3.5) circle (2pt);
	\coordinate[label={[align=center]below:{$T_{M', v}$} }] (A) at (13.5, -3.50);

	\begin{scope}[shift={(0,0), scale=0.69}]
\node[circle,fill = black,scale = 0.3] (A) at (0.5,0) {};
\node[circle,fill = black,scale = 0.3] (v) at (2,0) {};
\node[circle,fill = black,scale = 0.3] (B) at (4,2) {};
\node[circle,fill = black,scale = 0.3] (C) at (4,0.5) {};
\node[circle,fill = black,scale = 0.3] (E) at (4,-0.5) {};
\node[circle,fill = black,scale = 0.3] (F) at (4,-2) {};

		\draw (v) circle (.15cm);
		\draw (B) circle (.15cm);
		\draw (C) circle (.15cm);
		\draw (E) circle (.15cm);
		\draw (F) circle (.15cm);

\node at (4.0,2.4) {${v'}$};

\node at (4.4,2.0) {$\dots$};
\node at (4.4,1.9) {$\ddots$};
\node at (4.4,2.3) {$\iddots$};

\node at (4.4,0.5) {$\dots$};
\node at (4.4,0.4) {$\ddots$};
\node at (4.4,0.8) {$\iddots$};

\node at (4.4,0.5) {$\dots$};
\node at (4.4,0.4) {$\ddots$};
\node at (4.4,0.8) {$\iddots$};

\node at (4.4,-0.5) {$\dots$};
\node at (4.4,-0.6) {$\ddots$};
\node at (4.4,-0.2) {$\iddots$};

\node at (4.4,-2) {$\dots$};
\node at (4.4,-2.1) {$\ddots$};
\node at (4.4,-1.7) {$\iddots$};


\draw[very thick] (.5,0)--(2,0);
\draw (2,0)--(4,2);
\node at (3,0.7) {$\vdots$};
\node at (3,-0.5) {$\vdots$};
\draw (2,0)--(4,0.5);
\draw (2,0)--(4,-0.5);
\draw (2,0)--(4,-2);
\node at (1.4,0.15) {$m$};
\node at (2,-0.35) {$v$};
\node at (.5,-0.35) {$\bar v$};
\node at (3.3,1.5) {$1$};
\node at (3.3,0.5) {$1$};
\node at (3.3,-0.15) {$1$};
\node at (3.3,-1.1) {$1$};

\draw[->] (3,-1.5)--(3,-3.3);
\end{scope}
	\begin{scope}[shift={(4.7,0), scale=0.69}]
\node[circle,fill = black,scale = 0.3] (A) at (0.5,0) {};
\node[circle,fill = black,scale = 0.3] (v) at (2,0) {};
\node[circle,fill = black,scale = 0.3] (B) at (4,2) {};
\node[circle,fill = black,scale = 0.3] (C) at (4,0.5) {};
\node[circle,fill = black,scale = 0.3] (E) at (4,-0.5) {};
\node[circle,fill = black,scale = 0.3] (F) at (4,-2) {};

		\draw (v) circle (.15cm);
		\draw (B) circle (.15cm);
		\draw (C) circle (.15cm);
		\draw (E) circle (.15cm);
		\draw (F) circle (.15cm);

\node at (4.0,2.4) {${v'}$};
\node at (4.4,0.5) {$\dots$};
\node at (4.4,0.4) {$\ddots$};
\node at (4.4,0.8) {$\iddots$};

\node at (4.4,-0.5) {$\dots$};
\node at (4.4,-0.6) {$\ddots$};
\node at (4.4,-0.2) {$\iddots$};

\node at (4.4,-2) {$\dots$};
\node at (4.4,-2.1) {$\ddots$};
\node at (4.4,-1.7) {$\iddots$};

\draw[line width = 3pt] (.5,0)--(2,0);
\draw[very thick] (2,0)--(4,2);
\node at (3,0.7) {$\vdots$};
\node at (3,-0.5) {$\vdots$};
\draw[very thick] (2,0)--(4,0.5);
\draw (2,0)--(4,-0.5);
\draw (2,0)--(4,-2);
		\node at (1.1,0.3) {\scriptsize{$\ge M - 2mn$}};
\node at (2,-0.35) {$v$};
\node at (0.5,-0.35) {$\bar v$};
\node at (3.3,1.6) {$m$};
\node at (3.3,0.5) {$m$};
\node at (3.3,-0.15) {$?$};
\node at (3.3,-1.1) {$?$};
\draw[->] (2.7,-1.5)--(2.7,-3.3);
	\end{scope}
\end{tikzpicture}